\documentclass[11pt]{article}
\usepackage{geometry}
\usepackage{color}
\usepackage{float}
\usepackage{textcomp}
\usepackage{url}
\usepackage{amsthm}
\usepackage{amsmath}
\usepackage{amssymb}%
\usepackage{mdframed}
\usepackage{multirow}
\usepackage{changes}
\usepackage{hyperref}
\usepackage{comment}
\usepackage{makecell}

\usepackage{mathtools}
\usepackage{booktabs}    
\usepackage{subfigure} 
\usepackage{caption} 
\usepackage{multirow}
\usepackage{multicol}    
\usepackage{array}       
\usepackage{graphicx}
\graphicspath{{figuras/}}
\usepackage{empheq,amsmath}
\allowdisplaybreaks
\usepackage
[ 
lined, 
boxruled, 
commentsnumbered
]
{algorithm2e}

\DecMargin{0.1cm} 

\newtheorem{theorem}{Theorem}[section]
\newtheorem{lemma}[theorem]{Lemma}
\newtheorem{corollary}[theorem]{Corollary}

\newtheorem{remark}[theorem]{Remark}
\newtheorem{definition}[theorem]{Definition}

\newcommand{\bi}{\begin{itemize}}
\newcommand{\ei}{\end{itemize}}
\newcommand{\ba}{\begin{array}}
\newcommand{\ea}{\end{array}}

\begin{document}

\title{\textbf{A Finite-Difference Trust-Region Method for Convexly Constrained Smooth Optimization}}

\author{D\^an\^a Davar\thanks{ICTEAM Institute, UCLouvain, 1348 Louvain-la-Neuve, Belgium (dana.davar@uclouvain.be). This author was supported by the French Community of Belgium (FSR program).} 
\and Geovani Nunes Grapiglia\thanks{ICTEAM Institute, UCLouvain, 1348 Louvain-la-Neuve, Belgium (geovani.grapiglia@uclouvain.be). This author was partially supported by FRS-FNRS, Belgium (Grants CDR J.0081.23 and J.0094.26).} }

\date{June 30, 2026}

\maketitle

\begin{abstract}
We propose a derivative-free trust-region method based on finite-diffe\-rence gradient approximations for smooth optimization problems with convex constraints. For nonconvex problems, we establish a worst-case complexity bound of $\mathcal{O}\!\left(n\left(\frac{L}{\sigma}\epsilon\right)^{-2}\right)$ function evaluations for the method to reach an $\left(\frac{L}{\sigma}\epsilon\right)$-approximate stationary point, where $n$ is the number of variables, $L$ is the Lipschitz constant of the gradient, and $\sigma$ is a user-defined estimate of $L$. If the objective function is convex, the complexity to reduce the functional residual below $(L/\sigma)\epsilon$ is shown to be of $\mathcal{O}\!\left(n\left(\frac{L}{\sigma}\epsilon\right)^{-1}\right)$ function evaluations, while for Polyak--\L{}ojasiewicz functions on unconstrained domains, the bound further improves to $\mathcal{O}\left(n\log\left(\left(\frac{L}{\sigma}\epsilon\right)^{-1}\right)\right)$. Numerical experiments on benchmark problems with noise-free and noisy objective functions, as well as a model-fitting application, show the efficiency of the proposed method relative to state-of-the-art derivative-free solvers for unconstrained and bound-constrained problems.
\end{abstract}

\section{Introduction}\label{sec:1}

\subsection{Problem and Contributions}

We consider optimization problems of the form
\begin{equation}
\text{Minimize } f(x) \quad \text{subject to } x \in \Omega,
\label{eq:first}
\end{equation}
where $f:\mathbb{R}^n \to \mathbb{R}$ is assumed to be continuously differentiable with Lipschitz continuous gradient and bounded from below by $f_{low}$, and where $\Omega \subseteq \mathbb{R}^n$ is a nonempty closed convex set. However, we focus on the scenario in which $f\left(\,\cdot\,\right)$ is only accessible through a zeroth-order oracle, i.e., given $x$ we can only evaluate $f(x)$. Problems of this type arise in several applications, such as the calibration of PDE models \cite{zika}, bilevel optimization \cite{bilevel}, design of black-box attacks in deep neural networks \cite{adversarial}, and shape optimization \cite{elgeti}. Since gradients of the objective function are not readily available, \textit{Derivative-Free Optimization} (DFO) methods are required \cite{CSV2,AH,LMW}.

Over the past decades, significant progress has been made in the field of DFO, driven by the design of new algorithms with strong convergence and worst-case complexity guarantees, together with efficient software implementations. Among these recent developments, the TRFD method proposed in \cite{grapiglia} addresses problems of the form (\ref{eq:first}) for a composite objective function $f(x)=h(F(x))$, where $F:\mathbb{R}^{n}\to\mathbb{R}^{m}$ is a black-box function and $h:\mathbb{R}^{m}\to\mathbb{R}$ is a known convex Lipschitz continuous function, possibly nonsmooth (for example, $h(z)=\|z\|_{1}$). Specifically, TRFD is a trust-region method that constructs finite-difference approximations $A_k\in\mathbb{R}^{m\times n}$ of the Jacobian of $F(\,\cdot\,)$ at $x_k$, and computes a trial step by approximately solving the subproblem
\begin{equation}
\text{Minimize}\quad h(F(x_{k})+A_{k}d)\quad\text{subject to}\quad\|d\|_{p}\leq\Delta_{k},\,\,\text{and}\,\,x_{k}+d\in\Omega,
\label{eq:tr_subproblem}
\end{equation}
where the $p$-norm can be chosen accordingly to the function $h\left(\,\cdot\,\right)$. Thanks to combined update rules for the trust-region radius and the finite-difference stepsize, TRFD was shown to find $\epsilon$-approximate stationary points using at most $\mathcal{O}\left(n\epsilon^{-2}\right)$ evaluations of $F\left(\,\cdot\,\right)$ when the Jacobian of $F\left(\,\cdot\,\right)$ is Lipschitz continuous and the $p$-norm in (\ref{eq:tr_subproblem}) is chosen appropriately. Moreover, numerical experiments demonstrated that TRFD can outperform state-of-the-art DFO solvers for composite nonsmooth optimization problems.

Motivated by the encouraging performance of TRFD on composite nonsmooth optimization, we propose in this work a variant of it, called TRFD-S, for solving smooth optimization problems with convex constraints. The modifications leading to TRFD-S were designed to address two main drawbacks of the original TRFD algorithm. First, observe that when $F:\mathbb{R}^{n}\to\mathbb{R}$ is a real function and $h:\mathbb{R}\to\mathbb{R}$ is given by $h(z)=z$, we have $f(x)=h(F(x))=F(x)$. In this setting, subproblem (\ref{eq:tr_subproblem}) with $p=2$ reduces to the constrained minimization of an approximate first-order Taylor model of $f(x_{k}+d)$. Hence, a direct application of TRFD to smooth problems does not exploit second-order derivative information, making the method unlikely to be competitive with curvature-aware approaches such as NEWUOA \cite{powell} and BOBYQA \cite{bobyqa}. Secondly, for the smooth setting described above, we have $A_{k}=g_{k}^{T}$, where $g_{k}$ is a finite-difference approximation of $\nabla f(x_{k})$. In this case, to guide the iterates toward stationary points, TRFD relies on the exact computation of the approximate stationarity measure 
\begin{equation}
\eta_{\Delta_{\max}}(x_{k})=\frac{1}{\Delta_{\max}}\left(-\min_{s\in \Omega - \{x\} \atop \|s\| \leq \Delta_{\max}} \langle g_{k}, s\rangle\right)
\label{eq:eta_dfn}
\end{equation}
where $\Delta_{\max}$ is a user-defined upper bound on the trust-region radius. Computing the quantity in~(\ref{eq:eta_dfn}) amounts to solving a linear optimization problem over a convex set, which does not admit a closed-form solution for a general convex set $\Omega$. Therefore, although the theoretical analysis assumes that $\eta_{\Delta_{\max}}(x_k)$ can be computed exactly, its evaluation in practice typically relies on an iterative auxiliary solver, and is thus only approximate.

To address these issues, TRFD-S incorporates a quadratic term,
$\frac{1}{2}\langle H_k s,s\rangle$, into the local models, where
$H_k$ may be an approximation of the Hessian of the objective function at
$x_k$, constructed using only zeroth-order information. Moreover, TRFD-S
completely avoids the computation of the approximate stationarity measure
$\eta_{\Delta_{\max}}(x_k)$. These modifications, particularly the latter, require a new theoretical
analysis, as omitting $\eta_{\Delta_{\max}}(x_k)$ leads to worst-case
complexity guarantees for scaled, rather than absolute, accuracy
measures. In particular, we show that TRFD-S
requires at most
\[
\mathcal{O}\left(
n\left(\frac{\sigma}{L}\right)^2
\left(f(x_0)-f_{low}\right)
\epsilon^{-2}
\right)
\]
function evaluations to compute an
$\left(\frac{L}{\sigma}\epsilon\right)$-approximate stationary point
of $f\left(\,\cdot\,\right)$ over $\Omega$, where $L$ denotes the Lipschitz constant of
$\nabla f$ and $\sigma$ is a user-specified estimate of $L$. When $f\left(\,\cdot\,\right)$ is convex, we further prove that TRFD-S requires at most
\[
\mathcal{O}\left(
n\left(\frac{\sigma}{L}\right)
L\Delta_{\max}^{2}
\epsilon_f^{-1}
\right)
\]
function evaluations to compute an iterate $x_k$ satisfying $f(x_k)-f(x^*) \leq \frac{L}{\sigma}\epsilon_f$, where $x^{*}$ denotes the minimizer of $f\left(\,\cdot\,\right)$ over $\Omega$. Furthermore, if $f\left(\,\cdot\,\right)$ satisfies the Polyak--{\L}ojasiewicz inequality \cite{polyak}
with parameter $\mu>0$ and $\Omega=\mathbb{R}^n$, we establish the improved
complexity bound of
\[
\mathcal{O}\left(
n\left(\frac{L}{\mu}
\right)\log\left(
\left(\frac{\sigma}{L}\right)
\left(f(x_0)-f(x^*)\right)
\epsilon_f^{-1}
\right)
\right)
\]
function evaluations to achieve the same accuracy. Note that when $\sigma = L$, all of these complexity bounds match, up to an additional factor of $n$, the corresponding bounds for gradient descent on the same problem classes. This factor $n$ arises from the computation of finite-difference gradient approximations at each iteration.

Finally, we complement the theoretical developments with an extensive numerical study comparing TRFD-S with several derivative-free optimization solvers, including
NEWUOA \cite{powell,zhang_2023}, TRFD \cite{grapiglia}, DFQRM \cite{grapiglia2}, BOBYQA \cite{bobyqa,zhang_2023}, and NOMAD \cite{audet}. The experiments cover both unconstrained and constrained benchmark problems under exact and inexact function evaluations. We also compare TRFD-S with BOBYQA on an ODE parameter calibration problem. The results demonstrate a substantial improvement of TRFD-S over TRFD and show that TRFD-S is competitive with, and in several settings outperforms, leading derivative-free solvers such as NEWUOA and BOBYQA.

\subsection{Contents}

The paper is structured as follows. In Section \ref{sec:2}, we give the assumptions and auxiliary results. In Section \ref{sec:3} we present TRFD-S for problems with relaxable convex constraints and prove worst-case evaluation complexity bounds for nonconvex, convex and Polyak-Lojasiewicz objective functions. \textcolor{black}{We also propose an adaptation of TRFD-S to unrelaxable bound constraints in the latter section}. Finally, in Section \ref{sec:4}, we provide numerical results on benchmark problems for unconstrained and bound constraints sets, in addition to showing a model fitting application.

\subsection{Notations}

\textcolor{black}{In this paper, $\left\langle\,\cdot\,,\,\cdot\,\right\rangle$ denotes the scalar product between two vectors and $\|\,\cdot\,\|$ stands for the Euclidean norm of a vector, while $\Omega-\left\{x\right\}:=\left\{s \in \mathbb{R}^n:x+s\in \Omega\right\}$.}

\section{Assumptions and Auxiliary Results}\label{sec:2}

Through the paper, we will consider the following assumptions:
\begin{mdframed}
\textbf{A1.} $\Omega \subset \mathbb{R}^n$ is a \textcolor{black}{relaxable\footnotemark} nonempty closed convex set.
\\
\textbf{A2.} $f: \mathbb{R}^n \to \mathbb{R}$ is differentiable and its gradient $\nabla f$ is $L$-Lipschitz with respect to the Euclidean norm.
\end{mdframed}
\footnotetext{\textcolor{black}{Meaning that the objective function can be evaluated outside the domain $\Omega$ \textcolor{black}{(see, e.g., \cite{dw})}. The unrelaxable case (where the objective function can \textit{not} be evaluated outside $\Omega$) is considered in subsection \ref{sec:unrelaxable}.}}
In addition, the following definition of stationarity will be used.
\begin{definition}\label{first_def}
    A point $x^*\in\Omega$ is a stationary point of $f$ when $$\left\langle\nabla f(x^*), s\right\rangle \geq 0, \quad \forall s \in \Omega-\{x^*\}.$$
\end{definition}
\noindent Definition \ref{first_def} motivates the use of the following stationarity measure. Given $r>0$, let us denote \begin{equation}\label{psi}
    \psi_r(x)=\frac{1}{r}\left(-\min_{s\in \Omega - \{x\} \atop \|s\| \leq r} \langle \nabla f(x), s\rangle\right).
\end{equation}
\noindent The lemma below provides some properties on this stationarity measure.
\begin{lemma}\label{prop_psi}
    (\textcolor{black}{Theorems 12.1.5 and 12.1.6 in \cite{toint}, Lemmas 2.6 and 2.12 in \cite{grapiglia}}). Suppose that A1 and A2 hold, and let $\psi_r$ be defined by \eqref{psi}. Then,
    \vspace{1.5mm}
    \\
    {\color{black}(a) $\psi_{\left(\,\cdot\,\right)}\left(\,\cdot\,\right)$ is a continuous function in terms of $r$ and $x\in \Omega$, for $r>0$;} \\
    (b) $\psi_{r}(x) \geq 0, \quad \forall x \in \Omega$; \\
    (c) $\psi_{r}(x^{*})=0$ if, and only if, $x^*$ is a stationary point of $f$ in $\Omega$; \\
    \textcolor{black}{(d) $\psi_{r_2}(x) \leq \psi_{r_1}(x), \quad 0 < r_1 \leq r_2, \; \forall x\in\Omega$}.
\end{lemma}
\begin{remark}
    In the particular case where $\Omega=\mathbb{R}^n$, the stationarity measure $\psi_r(x)$ reduces to $\left\|\nabla f(x)\right\|$.
\end{remark}

\noindent In view of Lemma \ref{prop_psi}, we say that a point $x\in\Omega$ is an $\epsilon$-approximate stationary point of $f$ in $\Omega$ with respect to $r>0$, when $\psi_{r}(x)\leq\epsilon$.
\vspace{2mm}
\\
\noindent Since the gradient is supposed not to be accessible, the stationarity measure defined in \eqref{psi} will be approached by 
\begin{equation}\label{eta}
    \eta_r(x)=\frac{1}{r}\left(-\min_{s\in \Omega - \{x\} \atop \|s\| \leq r} \langle g, s\rangle\right),
\end{equation} where $g\in \mathbb{R}^n$ is an approximation to the gradient of $f$ at $x$.
\\
\\
\noindent The following lemma gives a bound for $\|\nabla f(x)-g\|$ when $g$ is a forward finite-difference approximation of $\nabla f(x)$.

\begin{lemma}\label{firstlem}
    \textcolor{black}{(Section 8.1 in \cite{NW} or (10.60) in \cite{quarteroni} for the univariate case)}. Suppose that A2 holds. Given $x \in \mathbb{R}^n$ and $\tau > 0$, let $g \in \mathbb{R}^n$ be defined by \begin{equation*}\label{fg}
        [g]_i=\frac{f(x+\tau e_i)-f(x)}{\tau}, \quad i=1, ...,n.
    \end{equation*} Then, \begin{equation*}
        {\|\nabla f(x)-g\|} \leq \frac{L}{2}\tau\sqrt{n}.
        \label{fdap}
    \end{equation*}
\end{lemma}
\begin{remark}
    On one hand, forward finite differences give an error bound of $\mathcal{O}(\tau_k)$ by evaluating $n$ times the function. On the other hand, central finite differences would lead to an error of $\mathcal{O}(\tau_k^2)$, but by requiring $2n$ evaluations to construct the approximation $g_k$.
\end{remark}
\begin{remark}
    In view of Lemma \ref{firstlem}, the stationarity measure $\psi_r(x)$ can be related to the approximate stationarity measure $\eta_r(x)$. Such connection is established in Lemma \ref{psieta}.
\end{remark}

\section{Trust-Region Method for Convexly Constrained Problems}\label{sec:3}

In what follows, we present TRFD-S, a derivative-free \textbf{T}rust-\textbf{R}egion method based on \textbf{F}inite \textbf{D}ifferences for \textbf{S}mooth problems with relaxable convex constraints, that is, problems in which function values can be computed at points outside the feasible set (see subsection \ref{sec:unrelaxable} for the adaptation to unrelaxable bound constraints). At the $k$-th iteration of TRFD-S, an approximation $g_k$ of $\nabla f(x_k)$ is built by using finite differences. Then, a step $d_k$ is computed by solving approximately a trust-region subproblem, where the model to minimize is given by $$m_k(d) = f(x_k)+\langle g_{k}, d\rangle + \frac{1}{2}\langle H_k d, d\rangle,$$ with $H_k$ being an approximation to the Hessian of $f$ at $x_k$, \textcolor{black}{which can be computed from approximate first-order information (see also Remark \ref{rem:2})}. After, we assess the quality of the step $d_k$. If $$\frac{f(x_k)-f(x_k+d_{k})}{m_k(0)-m_k(d_k)} \geq \alpha,$$ with $\alpha \in (0,1)$, then we define $x_{k+1} = x_k + d_k$, the trust-region radius can grow, and the finite-difference stepsize remains constant by defining $\tau_{k+1}=\tau_k$. Otherwise, the method sets $x_{k+1}=x_{k}$, the trust-region radius is halved, while the finite-difference stepsize is possibly reduced. The trust-region radius keeps decreasing until the step is accepted.

In Algorithm 1, we describe precisely the steps of our new method. Notice that in Step 0, the estimate $\sigma>0$ of $L$ is a user-defined parameter. As we can see in Tables \ref{tab:1}, \ref{tab:2} and \ref{tab:3}, its value can have a significant impact. If $\sigma$ overestimates $L$, then an approximate stationary point with tighter tolerance than $\epsilon$ is guaranteed, while the worst-case evaluation complexity is increased. In the case where $\sigma$ underestimates $L$, then the worst-case evaluation complexity decreases, while an approximate stationary point with larger tolerance than $\epsilon$ can only be guaranteed.
\vspace{0.3cm}
\\
\noindent\textbf{Algorithm 1:} \textbf{TRFD-S} for relaxable convex constraints
\\[0.2cm]
\noindent\textbf{Step 0.} Given a feasible set $\Omega$, an initial point $x_0 \in \Omega$, a parameter $\epsilon > 0$, an estimate $\sigma > 0$ of $L$, and a threshold $\alpha \in (0, 1)$ for accepting trial points, define
\begin{equation*} \tau_{0} = \dfrac{\epsilon}{\sigma\sqrt{n}}.
\end{equation*} Choose an initial trust-region radius $\Delta_0$ and an upper bound on the trust-region radii $\Delta_{\mathrm{max}}$ such that $\tau_0 \sqrt{n} \leq \Delta_0 \leq \Delta_{\mathrm{max}}$, and set $k:=0$.
\\[0.2cm]
\noindent\textbf{Step 1.} Construct $g_{k} \in \mathbb{R}^{n}$ with \begin{equation*}
    \left[g_{k}\right]_i = \frac{f(x_k+\tau_{k}e_i)-f(x_k)}{\tau_{k}}, \quad i=1,...,n,
\end{equation*}
and choose a nonzero symmetric matrix $H_k\in \mathbb{R}^{n\times n}$.
\\[0.2cm]
\noindent\textbf{Step 2.} Compute an approximate solution $d_k$ of the trust-region subproblem 
\begin{equation*}
\begin{aligned}
\min_{d\in\mathbb{R}^{n}} \quad & m_k(d) = f(x_k)+\langle g_{k}, d\rangle + \frac{1}{2}\langle H_k d, d\rangle\\
\textrm{s.t.} \quad & \|d\| \leq \Delta_{k}\\
& x_k+d \in \Omega
\end{aligned}
\end{equation*}
such that
\begin{equation}\label{sdcdf}
    m_k(0)-m_k(d_k) \geq \kappa\eta_{\Delta_{\mathrm{max}}}(x_k) \min\left\{\Delta_k, \frac{\eta_{\Delta_{\mathrm{max}}}(x_k)}{\|H_k\|}\right\},
    \end{equation}
where $\kappa \in (0,1)$ is a constant independent of $k$, \textcolor{black}{and $\eta_{\Delta_{\mathrm{max}}}(x_k)$ is the approximate stationarity measure \eqref{eta} with $r=\Delta_{\mathrm{max}}$ (see Remark \ref{rem:2} for discussing the guarantee to have \eqref{sdcdf}).}
\\[0.2cm]
\noindent\textbf{Step 3.} Compute \begin{equation}\label{ratio_generalset}
    \rho_{k}=\frac{f(x_k)-f(x_k+d_{k})}{m_k(0)-m_k(d_k)}.
\end{equation} 
If $\rho_{k}\geq\alpha$, define $x_{k+1}=x_{k}+d_{k}$, $\Delta_{k+1}=\min\left\{2\Delta_{k},\Delta_{\mathrm{max}}\right\}$, $\tau_{k+1}=\tau_{k}$, set $k:=k+1$ and go to Step 1.
\\[0.2cm]
\noindent\textbf{Step 4} Define $x_{k+1}=x_k$ and $\Delta_{k+1}=\frac{1}{2}\Delta_k$. If $\tau_k\sqrt{n} \leq \Delta_{k+1}$, define $\tau_{k+1}=\tau_{k}$, $g_{k+1}=g_{k}$, $H_{k+1}=H_k$, set $k:=k+1$ and go to Step 2. Otherwise, define $\tau_{k+1} = \frac{1}{2}\tau_k$, set $k:=k+1$ and go to Step 1.
\vspace{0.2cm}
\begin{remark}\label{rem:2}
    In contrast with TRFD \cite{grapiglia} when applied to smooth optimization (i.e., the case $F:\mathbb{R}^n\to\mathbb{R}$ and $h(z)=z,\forall z \in \mathbb{R}$), TRFD-S approximates second-order information through $H_k$, which may be constructed from approximate first-order information using various strategies, such as the safeguarded BFGS update outlined in subsection~\ref{perfo_unc}. In addition, TRFD requires \textcolor{black}{computing} two trust-region subproblems per iteration, while TRFD-S only needs to solve one. Finally, TRFD assumes that the model decrease is at least proportional to a fraction of the exact model decrease, while TRFD-S assumes inequality \eqref{sdcdf}. In the case where $\Omega = \mathbb{R}^n$, condition \eqref{sdcdf} is naturally satisfied by any step at least as good as the Cauchy step. In the case where $\Omega \neq \mathbb{R}^n$ and A1 holds, condition \eqref{sdcdf} is also guaranteed by any step at least as good as the Generalized Cauchy step, which can be computed by Algorithm 12.2.2 in \cite{toint}. We provide a proof for this claim in the Appendix \ref{sec:appendix}. Notice that despite condition \eqref{sdcdf}, TRFD-S never computes $\eta_{\Delta_{\mathrm{max}}}(x_k)$.
\end{remark}
\noindent In TRFD-S, we have the following sets of iterations:
\vspace{0.1cm}
\begin{enumerate}
\item\textbf{Successful iterations} ($\mathcal{S}$): those where $\rho_{k}\geq\alpha$.
\item\textbf{Unsuccessful iterations of type I} ($\mathcal{U}^{(1)}$): those where $\rho_{k}<\alpha$ and $\tau_{k}\sqrt{n}\leq\Delta_{k+1}$.
\item\textbf{Unsuccessful iterations of type II} ($\mathcal{U}^{(2)}$): those where $\rho_{k}<\alpha$ and $\tau_{k}\sqrt{n}>\Delta_{k+1}$.
\end{enumerate}
\begin{remark}
    {\color{black}By Step 4 of TRFD-S, we see that the finite-difference stepsize is only reduced when $k \in \mathcal{U}^{(2)}$, i.e., when there is a strong evidence that the current value does not provide sufficiently accurate gradient approximation to allow the sufficient decrease of the objective function.}
\end{remark}
\noindent The lemma below shows that the finite-difference stepsize $\tau_{k}$ is always bounded from above by $\Delta_{k}/\sqrt{n}$.
\begin{lemma}
\label{lem:3.1}
Given $T \geq 1$, let $\left\{\tau_{k}\right\}_{k=0}^T$ and $\left\{\Delta_{k}\right\}_{k=0}^T$ be generated by TRFD-S.
Then
\begin{equation}
\tau_{k}\sqrt{n}\leq\Delta_{k}, \quad \text{for} \quad k=0,...,T.
\label{eq:3.1}
\end{equation}
\end{lemma}
\begin{proof}
Let us work through an induction argument. By Step 0 of TRFD-S, we have that \eqref{eq:3.1} holds for $k=0$. By assuming \eqref{eq:3.1} to be true for some $k \in \{0,...,T-1\}$, let us show that \eqref{eq:3.1} also holds for $k+1$. With our sets of iterations, we have three possible cases:
\\[0.2cm]
\noindent\textbf{Case I}: $k\in\mathcal{S}$.
\\[0.2cm]
By Step 3 of TRFD-S, we have $\tau_{k+1}=\tau_{k}$ and $\Delta_{k+1}\geq\Delta_{k}$. Thus, by the induction assumption,
\begin{equation*}
\tau_{k+1}\sqrt{n}=\tau_{k}\sqrt{n}\leq\Delta_{k}\leq\Delta_{k+1},
\end{equation*}
which means that \eqref{eq:3.1} is true for $k+1$.
\\[0.2cm]
\noindent\textbf{Case II}: $k\in\mathcal{U}^{(1)}$.
\\[0.2cm]
By Step 4 of TRFD-S, we have $\tau_{k+1}=\tau_{k}$ and $\tau_k\sqrt{n} \leq \Delta_{k+1}$. Then,
\begin{equation*}
    \tau_{k+1}\sqrt{n} = \tau_k\sqrt{n}\leq \Delta_{k+1},
\end{equation*}
so \eqref{eq:3.1} is true for $k+1$.
\\[0.2cm]
\noindent\textbf{Case III}: $k\in\mathcal{U}^{(2)}$.
\\[0.2cm]
By Step 4 of TRFD-S, we have $\tau_{k+1} = \frac{1}{2}\tau_k$ and $\Delta_{k+1} = \frac{1}{2}\Delta_k$. Thus, by using the induction assumption, we have
\begin{equation*}
    \tau_{k+1}\sqrt{n} = \frac{1}{2}\tau_k\sqrt{n} \leq \frac{1}{2}\Delta_k = \Delta_{k+1},
\end{equation*}
that is, \eqref{eq:3.1} is true for $k+1$, which concludes the proof.
\end{proof}
\noindent In view of Lemmas \ref{firstlem} and \ref{lem:3.1}, the finite-difference approximation $g_{k}$ in TRFD-S satisfies \textcolor{black}{(see Assumption 4 in Conejo et al. \cite{karas})}
\begin{equation}\label{fg_delta}
    \|\nabla f(x_{k})-g_{k}\|\leq\dfrac{L}{2}\Delta_{k},\quad\forall k \geq 0.
\end{equation}
\noindent Using the previous inequality, the next lemma proves that if the trust-region radius is sufficiently small, then $k\in \mathcal{S}$.
\begin{lemma}
\label{lem:3.2}
Suppose that A1 and A2 hold, and let $x_k$ be generated by TRFD-S. If \begin{equation}
\Delta_{k}\leq \dfrac{(1-\alpha)\kappa\eta_{\Delta_\mathrm{max}}(x_{k})}{2L+\|H_k\|},
\label{eq:3.2}
\end{equation}
then $k\in \mathcal{S}$, where $\kappa$ is the constant in \eqref{sdcdf}.
\end{lemma}
\begin{proof}
By \eqref{ratio_generalset}, A2, \eqref{sdcdf} and \eqref{fg_delta}, we have
\begin{align*}
    1-\rho_k &= \frac{m_k(0)-m_k(d_k)-(f(x_k)-f(x_k+d_k))}{m_k(0)-m_k(d_k)} \\
    &\leq \frac{f(x_k+d_k)-f(x_k)- \langle \nabla f(x_k), d_k \rangle + \langle \nabla f(x_k), d_k \rangle - \langle g_k, d_k \rangle - \frac{1}{2}\langle H_kd_k,d_k \rangle}{m_k(0)-m_k(d_k)} \\
    &\leq \frac{\left|f(x_k+d_k)-f(x_k)-\langle \nabla f(x_k), d_k \rangle \right| + \left| \langle \nabla f(x_k), d_k \rangle - \langle g_k, d_k \rangle \right| + \left| \frac{1}{2}\langle H_kd_k,d_k \rangle \right|}{\kappa\eta_{\Delta_\mathrm{max}}(x_k)\min\left\{\Delta_k, \frac{\eta_{\Delta_\mathrm{max}}(x_k)}{\|H_k\|}\right\}} \\
    &\leq \frac{\frac{L}{2}\|d_k\|^2 + \| \nabla f(x_k) - g_k\|\|d_k\| + \frac{1}{2}\|H_k\|\|d_k\|^2}{\kappa\eta_{\Delta_\mathrm{max}}(x_k)\min\left\{\Delta_k, \frac{\eta_{\Delta_\mathrm{max}}(x_k)}{\|H_k\|}\right\}} \\
    &\leq \frac{\frac{L}{2}\Delta_k^2 + \frac{L}{2}\Delta_k^2 + \frac{\|H_k\|}{2}\Delta_k^2}{\kappa\eta_{\Delta_\mathrm{max}}(x_k)\min\left\{\Delta_k, \frac{\eta_{\Delta_\mathrm{max}}(x_k)}{\|H_k\|}\right\}} \\
    &< \frac{(2L + \|H_k\|)\Delta_k^2}{\kappa\eta_{\Delta_\mathrm{max}}(x_k)\min\left\{\Delta_k, \frac{\eta_{\Delta_\mathrm{max}}(x_k)}{\|H_k\|}\right\}}. 
\end{align*}
Since $\alpha \in (0,1)$, $\kappa \in (0,1)$ and $2L > 0$, by \eqref{eq:3.2} we get $$\Delta_k < \frac{\eta_{\Delta_\mathrm{max}}(x_k)}{\|H_k\|}.$$ So, it follows that $$1-\rho_k \leq \frac{(2L + \|H_k\|)\Delta_k}{\kappa\eta_{\Delta_\mathrm{max}}(x_k)}.$$
Finally, by \eqref{eq:3.2} we have $$1-\rho_k \leq 1-\alpha.$$ Thus, we get $\rho_k \geq \alpha$, meaning that $k\in\mathcal{S}$, which concludes the proof.
\end{proof}
\noindent The next lemma bounds the error $\left|\psi_{\Delta_{\mathrm{max}}}(x_k)-\eta_{\Delta_{\mathrm{max}}}(x_k)\right|$ with the same quantity as for $\|\nabla f(x_k)-g_k\|$.

\begin{lemma}\label{psieta}
    (Lemma 2.10 in \cite{grapiglia}). Suppose that A1 and A2 hold, and let $x_k$ be generated by TRFD-S. Then, \begin{equation}\label{eq13}
        \left|\psi_{\Delta_{\mathrm{max}}}(x_k) - \eta_{\Delta_{\mathrm{max}}}(x_k)\right| \leq \frac{L}{2}\tau_k\sqrt{n}.
    \end{equation}
\end{lemma}
\noindent By using the tolerance $\left(\frac{L}{\sigma}\epsilon\right)$, the following lemma provides a lower bound on the approximate stationarity measure $\eta_{\Delta_{\mathrm{max}}}(x_k)$.
\begin{lemma}\label{eta_k_0}
    Suppose that A1 and A2 hold, and let $x_k$ be generated by TRFD-S. If
    \begin{equation}\label{psi_L_sigma}
        \psi_{\Delta_{\mathrm{max}}}(x_k) > \left(\frac{L}{\sigma}\epsilon\right),
    \end{equation}
    then \begin{equation}\label{eta_eps}
        \eta_{\Delta_{\mathrm{max}}}(x_k) > \frac{1}{2}\left(\frac{L}{\sigma}\epsilon\right).
    \end{equation}
\end{lemma}
\begin{proof}
By the update rules of TRFD-S, we have
\begin{equation}\label{tau_k_tau_0}
    \tau_k \leq \tau_0, \quad \forall k \geq 0.
\end{equation}
Then, by Lemma \ref{psieta}, \eqref{tau_k_tau_0}, the definition of $\tau_0$ in Step 0 of TRFD-S and \eqref{psi_L_sigma}, it follows that 
\begin{align*}
    \psi_{\Delta_{\mathrm{max}}}(x_k)
    &\leq |\psi_{\Delta_{\mathrm{max}}}(x_k) - \eta_{\Delta_{\mathrm{max}}}(x_k)| + \eta_{\Delta_{\mathrm{max}}}(x_k) \leq \frac{L}{2}\tau_k\sqrt{n} + \eta_{\Delta_{\mathrm{max}}}(x_k) \\ &\leq \frac{L}{2}\tau_0\sqrt{n} + \eta_{\Delta_{\mathrm{max}}}(x_k) = \frac{L}{2}\frac{\epsilon}{\sigma} + \eta_{\Delta_{\mathrm{max}}}(x_k) \\
    &< \frac{1}{2}\psi_{\Delta_{\mathrm{max}}}(x_k) + \eta_{\Delta_{\mathrm{max}}}(x_k).
\end{align*}
Thus, we get
\begin{equation}\label{eta_psi}
    \eta_{\Delta_{\mathrm{max}}}(x_k) > \frac{1}{2}\psi_{\Delta_{\mathrm{max}}}(x_k).
\end{equation}
Therefore, by \eqref{psi_L_sigma}, it follows that \eqref{eta_eps} is true.
\end{proof}
\noindent Now let us consider the following assumption on the matrix $H_k$:
\vspace{0.2cm}
\begin{mdframed}
\textbf{A3.} There exists a nonzero positive constant $M$, \textcolor{black}{independent of $k$}, such that $\|H_k\| \leq M$, for all $k \geq 0$.
\end{mdframed}
\vspace{0.2cm}
\noindent The next lemma gives a lower bound on the trust-region radius $\Delta_k$.
\begin{lemma}
\label{lem:3.3}
Suppose that A1-A3 hold. Given $T\geq 1$, let $\left\{\Delta_{k}\right\}_{k=0}^{T}$ be generated by TRFD-S. If
\begin{equation*}
    \psi_{\Delta_{\mathrm{max}}}(x_k) > \left(\frac{L}{\sigma}\epsilon\right), \quad \text{for} \quad k=0,...,T-1,
\end{equation*}
then
\begin{equation}
\Delta_{k}\geq\dfrac{(1-\alpha)\kappa}{8L+4M}\left(\frac{L}{\sigma}\epsilon\right)\equiv\Delta_{\min}(\epsilon),\quad\text{for}\quad k=0,\ldots,T,
\label{delta}
\end{equation}
where $\kappa$ is the constant in \eqref{sdcdf}.
\end{lemma}
\begin{proof}
For $k=0$, by Step 0 of TRFD-S, since $\alpha \in (0,1)$, $\kappa \in (0,1)$ and $\frac{L}{8L+4M} \in (0,1)$, we have
\begin{equation*}
    \Delta_0 \geq \tau_0\sqrt{n} = \frac{\epsilon}{\sigma} > \frac{ (1-\alpha)\kappa}{8L+4M}\left(\frac{L}{\sigma}\epsilon\right) = \Delta_{\min}(\epsilon).
\end{equation*}
So, \eqref{delta} is true for $k=0$. Now, let us assume that \eqref{delta} is true for some $k\in\left\{0,\ldots,T-1\right\}$. On one hand, if $\rho_k \geq \alpha$, then Step 3 of TRFD-S and the induction assumption imply that
\begin{equation*}
\Delta_{k+1} \geq \Delta_{k}\geq{\Delta}_{\min}(\epsilon).
\end{equation*}
On the other hand, if $\rho_k < \alpha$, then by Lemmas \ref{lem:3.2}, \ref{eta_k_0} and A3, we must have
\begin{equation}\label{delta_larger}
    \Delta_k > \frac{(1-\alpha)\kappa}{2L+M}\,\frac{1}{2}\left(\frac{L}{\sigma}\epsilon\right),
\end{equation}
since otherwise we would have $\rho_k \geq \alpha$, contradicting the assumption that $\rho_k < \alpha$. Then, by the update rule in Step 4 of TRFD-S and \eqref{delta_larger}, it follows that
\begin{equation*}
\Delta_{k+1}=\dfrac{1}{2}\Delta_{k}>\dfrac{(1-\alpha)\kappa}{8L+4M}\left(\frac{L}{\sigma}\epsilon\right) = \Delta_{\min}(\epsilon),
\end{equation*}
which shows that (\ref{delta}) is true.
\end{proof}

\subsection{Worst-Case Complexity Bound for Nonconvex Problems}

Given $j\in\left\{0,1,\ldots\right\}$, let
\begin{eqnarray*}
\mathcal{S}_{j}&=&\left\{0,\ldots,j\right\}\cap\mathcal{S},\\
\mathcal{U}_{j}^{(1)}&=&\left\{0,\ldots,j\right\}\cap\mathcal{U}^{(1)},\\
\mathcal{U}_{j}^{(2)}&=&\left\{0,\ldots,j\right\}\cap\mathcal{U}^{(2)}.
\end{eqnarray*}
Also, let 
\begin{equation}
    T_{g}(\epsilon)=\inf\left\{k\in\mathbb{N}\,:\,\psi_{\Delta_{\mathrm{max}}}(x_{k})\leq \left(\frac{L}{\sigma}\epsilon\right)\right\}
    \label{eq:hitting}
\end{equation}
be the first iteration index reaching an $\left(\frac{L}{\sigma}\epsilon\right)$-approximate stationary point of $f$ in $\Omega$, if it exists. Our goal is to obtain a finite upper bound for $T_{g}(\epsilon)$. If $T_{g}(\epsilon)\geq 1$, it follows from the notation above that
\begin{equation}
T_{g}(\epsilon) = \left|\mathcal{S}_{T_{g}(\epsilon)-1}\right| + \left|\mathcal{U}_{T_{g}(\epsilon)-1}^{(1)} \cup \mathcal{U}_{T_{g}(\epsilon)-1}^{(2)}\right|.
\label{eq:motivation}
\end{equation}
In the next two lemmas, we will provide upper bounds for the two terms in \eqref{eq:motivation}. To that end, let us consider the following additional assumption:
\vspace{0.2cm}
\begin{mdframed}
\textbf{A4.} There exists $f_{low}\in\mathbb{R}$ such that $f(x)\geq f_{low}$, for all $x\in\Omega$.
\end{mdframed}
\vspace{0.2cm}
The next lemma provides an upper bound on $\left|\mathcal{S}_{T_{g}(\epsilon)-1}\right|$.
\begin{lemma}
\label{lem:3.4bis}
Suppose that A1-A4 hold, and assume that $T_g(\epsilon)\geq 1$. Then
\begin{equation*}
\left|\mathcal{S}_{T_g(\epsilon)-1}\right|\leq \frac{(16L+8M)(f(x_0)-f_{low})}{\alpha(1-\alpha)\kappa^2}\left(\frac{L}{\sigma}\epsilon\right)^{-2},
\end{equation*}
where $\kappa$ is the constant in \eqref{sdcdf}.
\end{lemma}
\begin{proof}
Given $k \in \mathcal{S}_{T_g(\epsilon)-1}$, by \eqref{ratio_generalset}, \eqref{sdcdf}, Lemmas \ref{eta_k_0}, \ref{lem:3.3} and A3, we have \begin{align*}
    f(x_k)-f(x_{k+1})
    &\geq \alpha\kappa\eta_{\Delta_\mathrm{max}}(x_k) \min\left\{\Delta_k, \frac{\eta_{\Delta_\mathrm{max}}(x_k)}{\|H_k\|}\right\} \\
    &> \frac{\alpha\kappa}{2}\left(\frac{L}{\sigma}\epsilon\right) \min\left\{\dfrac{(1-\alpha)\kappa}{8L+4M}\left(\frac{L}{\sigma}\epsilon\right), \frac{1}{2M}\left(\frac{L}{\sigma}\epsilon\right)\right\}.
\end{align*}
Since $\alpha \in (0,1)$, $\kappa \in (0,1)$ and $L > 0$, it follows that \begin{equation}\label{k_in_S_A}
    f(x_k)-f(x_{k+1}) \geq \frac{\alpha(1-\alpha)\kappa^2}{16L+8M}\left(\frac{L}{\sigma}\epsilon\right) ^2,\quad\text{when}\,\,k\in\mathcal{S}_{T_g(\epsilon)-1}.
\end{equation} Notice that when $k\notin\mathcal{S}_{T_g(\epsilon)-1}$, then $f(x_{k})=f(x_{k+1})$. So, by A4 and \eqref{k_in_S_A} we get
\begin{eqnarray*}
f(x_{0})-f_{low}&\geq& f(x_{0})-f(x_{T_g(\epsilon)})= \sum_{k=0}^{T_g(\epsilon)-1}f(x_{k})-f(x_{k+1})\\ 
&=& \sum_{k\,\in\,\mathcal{S}_{T_g(\epsilon)-1}}f(x_{k})-f(x_{k+1})\\
&\geq &\left|\mathcal{S}_{T_g(\epsilon)-1}\right|\frac{\alpha(1-\alpha)\kappa^2}{16L+8M}\left(\frac{L}{\sigma}\epsilon\right)^2,
\end{eqnarray*}
which concludes the proof.
\end{proof}
\vspace{0.2cm}
\noindent The lemma below provides an upper bound on  $\left|\mathcal{U}_{T_{g}(\epsilon)-1}^{(1)} \cup \mathcal{U}_{T_{g}(\epsilon)-1}^{(2)}\right|$.

\begin{lemma}
\label{lem:3.5}
Suppose that A1-A3 hold, and assume that $T_g(\epsilon)\geq 1$. If $T\in\left\{1,\ldots,T_{g}(\epsilon)\right\}$, then
\begin{equation}
\left|\mathcal{U}_{T-1}^{(1)} \cup \mathcal{U}_{T-1}^{(2)}\right| \leq \log_{2}\left(\frac{(8L+4M)\Delta_{0}}{(1-\alpha)\kappa}\left(\frac{L}{\sigma}\epsilon\right)^{-1}\right) + |\mathcal{S}_{T-1}|,
\label{eq:3.13}
\end{equation}
where $\kappa$ is the constant in \eqref{sdcdf}.
\end{lemma}
\begin{proof}
By the update rules for $\Delta_{k}$ in TRFD-S, we have
\begin{eqnarray}
\Delta_{k+1}&\leq & 2\Delta_{k},\quad \text{if}\quad k\in\mathcal{S}_{T-1},\label{eq:3.14}\\
\Delta_{k+1}&=&\frac{1}{2}\Delta_{k},\quad \text{if}\quad k\in \mathcal{U}_{T-1}^{(1)} \cup \mathcal{U}_{T-1}^{(2)}\label{eq:3.16}.
\end{eqnarray}
In addition, by Lemma \ref{lem:3.3} we have
\begin{equation}
\Delta_{k}\geq \Delta_{\min}(\epsilon),\quad\text{for}\quad k=0,\ldots,T,\label{eq:3.17}
\end{equation}
where $\Delta_{\min}(\epsilon)$ is defined in (\ref{delta}). So, in view of (\ref{eq:3.14})-(\ref{eq:3.17}), it follows that
\begin{equation*}
2^{\left|\mathcal{S}_{T-1}\right| - \left|\mathcal{U}_{T-1}^{(1)} \cup \,\mathcal{U}_{T-1}^{(2)}\right|}\Delta_{0} \geq \Delta_{T}\geq \Delta_{\min}(\epsilon),
\end{equation*}
which gives
\begin{equation*}
    2^{\left|\mathcal{S}_{T-1}\right| - \left|\mathcal{U}_{T-1}^{(1)} \cup \,\mathcal{U}_{T-1}^{(2)}\right|} \geq \frac{\Delta_{\min}(\epsilon)}{\Delta_{0}}.
\end{equation*}
Then, taking the logarithm on both sides, we get
\begin{equation*}
    \left|\mathcal{S}_{T-1}\right| - \left|\mathcal{U}_{T-1}^{(1)} \cup \mathcal{U}_{T-1}^{(2)}\right| \geq \log_2\left(\frac{\Delta_{\min}(\epsilon)}{\Delta_{0}}\right),
\end{equation*}
which is equivalent to
\begin{equation}\label{us}
\left|\mathcal{U}_{T-1}^{(1)} \cup \mathcal{U}_{T-1}^{(2)}\right| \leq \log_{2}\left(\frac{\Delta_0}{\Delta_{\min}(\epsilon)}\right) + \left|\mathcal{S}_{T-1}\right|.
\end{equation}
Therefore, by the definition of $\Delta_{\min}(\epsilon)$ in \eqref{delta}, we conclude that \eqref{eq:3.13} is true.
\end{proof}
\noindent Combining the previous results, we obtain the following worst-case iteration complexity bound of TRFD-S to find an $\left(\frac{L}{\sigma}\epsilon\right)$-approximate stationary point of $f$ in $\Omega$.
\begin{theorem}
\label{thm:3.1}
Suppose that A1-A4 hold, and let $T_{g}(\epsilon)$ be defined by \eqref{eq:hitting}. Then
\small
\begin{equation}
T_{g}(\epsilon) \leq \frac{(32L+16M)(f(x_0)-f_{low})}{\alpha(1-\alpha)\kappa^2}\left(\frac{L}{\sigma}\epsilon\right)^{-2} + \log_{2}\left(\frac{(8L+4M)\Delta_{0}}{(1-\alpha)\kappa}\left(\frac{L}{\sigma}\epsilon\right)^{-1}\right) + 1,
\label{eq:3.23}
\end{equation}
\normalsize
where $\kappa$ is the constant in \eqref{sdcdf}.
\end{theorem}
\begin{proof}
If $T_{g}(\epsilon) \leq 1$, then we have that (\ref{eq:3.23}) is true. Let us assume that $T_{g}(\epsilon) \geq 2$. By (\ref{eq:motivation}),
\begin{equation*}
T_{g}(\epsilon) = \left|\mathcal{S}_{T_{g}(\epsilon)-1}\right| + \left|\mathcal{U}_{T_{g}(\epsilon)-1}^{(1)} \cup \mathcal{U}_{T_{g}(\epsilon)-1}^{(2)}\right|.
\end{equation*}
Then, (\ref{eq:3.23}) follows from Lemmas \ref{lem:3.4bis} and \ref{lem:3.5}.
\end{proof}
\noindent Since each iteration of TRFD-S requires at most $(n+1)$ evaluations of $f$, from Theorem \ref{thm:3.1} we obtain the following upper bound on the number of function evaluations required by TRFD-S to find an $\left(\frac{L}{\sigma}\epsilon\right)$-approximate stationary point of $f$ in $\Omega$.
\vspace{0.2cm}
\begin{corollary}\label{cor:1}
Suppose that A1-A4 hold, and let $FE_{T_{g}(\epsilon)}$ be the number of function evaluations executed by TRFD-S up to the $(T_{g}(\epsilon)-1)$-st iteration. Then
\small
\begin{equation}\label{cor:3.12}
FE_{T_{g}(\epsilon)} \leq (n+1)\left[\frac{(32L+16M)(f(x_0)-f_{low})}{\alpha(1-\alpha)\kappa^2}\left(\frac{L}{\sigma}\epsilon\right)^{-2} + \log_{2}\left(\frac{(8L+4M)\Delta_{0}}{(1-\alpha)\kappa}\left(\frac{L}{\sigma}\epsilon\right)^{-1}\right) + 1\right].
\end{equation}
\normalsize
\label{cor:3.1}    
\end{corollary}
\noindent In view of \eqref{cor:3.12}, TRFD-S needs no more than
\begin{equation}\label{cor_nc}
    \mathcal{O}\left(n\left(\frac{\sigma}{L}\right)^{2}L\left(f(x_0)-f_{low}\right)\epsilon^{-2}\right)
\end{equation}
function evaluations to reach a point $x_{k} \in \Omega$ such that $\psi_{\Delta_{\mathrm{max}}}(x_{k})\leq \frac{L}{\sigma}\epsilon$. Therefore, in the case where the user-defined parameter $\sigma$ equals the Lipschitz constant $L$, we get a worst-case evaluation complexity of
\begin{equation}\label{sigma_L}
    \mathcal{O}\left(n L\left(f(x_0)-f_{low}\right)\epsilon^{-2}\right)
\end{equation}
to satisfy $\psi_{\Delta_{\mathrm{max}}}(x_{k})\leq \epsilon$. Otherwise, when $\sigma \neq L$, Table \ref{tab:1} shows the different impacts of $\sigma$.
\begin{table}[h]
\centering
\begin{tabular}{|c|c|c|}
\hline
\textbf{Value of $\sigma$} & \textbf{Impact on \eqref{cor_nc}} & \makecell{\textbf{Impact on the target}\\ \textbf{accuracy $\left(\frac{L}{\sigma}\epsilon\right)$}} \\ \hline
$\sigma<L$ & \makecell{\eqref{cor_nc} lower than \eqref{sigma_L} \\ by a quadratic factor $\left(\frac{\sigma}{L}\right)^2$} & \makecell{Accuracy weaker than $\epsilon$ \\ by a factor $\left(\frac{L}{\sigma}\right)$} \\ \hline
$\sigma>L$ & \makecell{\eqref{cor_nc} larger than \eqref{sigma_L} \\ by a quadratic factor $\left(\frac{\sigma}{L}\right)^2$} & \makecell{Accuracy stricter than $\epsilon$ \\ by a factor $\left(\frac{L}{\sigma}\right)$} \\ \hline
\end{tabular}
\caption{Impacts of the user-defined parameter $\sigma$ for nonconvex problems}
\label{tab:1}
\end{table}
\begin{remark}
    In TRFD-S, if the user-defined parameter $\sigma$ underestimates $L$, then we can only guarantee that $\psi_{\Delta_{\mathrm{max}}}(x_{T_g(\epsilon)}) \leq \frac{L}{\sigma}\epsilon$, where $\frac{L}{\sigma}\epsilon$ can be significantly larger than $\epsilon$.
\end{remark}
\begin{remark}
    By Lemma \ref{prop_psi} (d), given $\sigma > 0$ and $\Delta_{\mathrm{max}}$, taking a larger upper bound $r_2 \geq \Delta_{\mathrm{max}}$ will lead to a tighter achieved accuracy:
\begin{equation*}
    \psi_{r_2}(x_{T_g(\epsilon)}) \leq \psi_{\Delta_{\mathrm{max}}}(x_{T_g(\epsilon)}) \leq \left(\frac{L}{\sigma}\epsilon\right), \quad 0<\Delta_{\mathrm{max}}\leq r_2.
\end{equation*}
\end{remark}
\subsection{Worst-Case Complexity Bound for Convex Problems}

Let us consider two additional assumptions:
\vspace{2mm}
\begin{mdframed}
\textbf{A5.} $f$ is convex.
\\[0.2cm]
\textbf{A6.} $f$ has a global minimizer $x^{*}$ in $\Omega$, and $D_{0}\equiv\sup_{x\in\mathcal{L}_{f}(x_{0})}\left\{\|x-x^{*}\|\right\}<+\infty$, for $\mathcal{L}_{f}(x_{0})=\left\{x\in\Omega\,:\,f(x)\leq f(x_{0})\right\}$.
\end{mdframed}
\vspace{2mm}
\noindent The lemma below establishes the relationship between the stationarity measure and the functional residual when the reference radius $r$ is sufficiently large.
\vspace{0.2cm}
\begin{lemma}
    \label{lem:3.2.1}
    (Lemma 3.14 in \cite{grapiglia}). Suppose that A1, A2, A5 and A6 hold, and let $x_k\in\mathcal{L}_{f}(x_{0})$. If $r\geq D_{0}$, then
    \begin{equation*}
        \psi_{r}(x_k) \geq \frac{f(x_k)-f(x^{*})}{r}.
    \end{equation*}
\end{lemma}

\noindent The next lemma provides a lower bound on the approximate stationarity measure $\eta_{\Delta_{\mathrm{max}}}(x_{k})$ in terms of the functional residual.
\begin{lemma}
    \label{lem:3.2.2}
    Suppose that A1, A2, A5 and A6 hold, and let $x_k$ be generated by TRFD-S. If $\Delta_{\mathrm{max}}\geq D_{0}$ and
    \begin{equation}\label{resid_eps}
        f(x_k) - f(x^*) > \Delta_{\mathrm{max}}\left(\frac{L}{\sigma}\epsilon\right),
    \end{equation}
    then
    \begin{equation}\label{eta_resid}
        \eta_{\Delta_{\mathrm{max}}}(x_{k}) > \dfrac{f(x_{k})-f(x^{*})}{2\Delta_{\mathrm{max}}}.
    \end{equation}
\end{lemma}
\begin{proof}
By Lemma \ref{lem:3.2.1} and \eqref{resid_eps}, it follows that
\begin{equation}\label{psi_str}
    \psi_{\Delta_{\mathrm{max}}}(x_k) > \left(\frac{L}{\sigma}\epsilon\right).
\end{equation}
Therefore, by \eqref{eta_psi} and Lemma \ref{lem:3.2.1}, we obtain \eqref{eta_resid}, which concludes the proof.
\end{proof}
\noindent Next, we establish an upper bound for $\frac{f(x_{k})-f(x^{*})}{\Delta_{k}}$.
\begin{lemma}
    \label{lem:3.2.3}
    Suppose that A1-A3, A5 and A6 hold. Given $T \geq 1$, let $\left\{x_{k}\right\}_{k=0}^{T}$ and $\left\{\Delta_{k}\right\}_{k=0}^{T}$ be generated by TRFD-S. If $\Delta_{\mathrm{max}}\geq D_{0}$ and
    \begin{equation*}
        f(x_k) - f(x^*) > \Delta_{\mathrm{max}}\left(\frac{L}{\sigma}\epsilon\right), \quad \text{for} \quad k=0,...,T-1,
    \end{equation*}
    then
    \begin{equation}
    \left(\frac{1}{\Delta_{k}}\right)(f(x_{k})-f(x^{*}))\leq\max\left\{\left(\frac{1}{\Delta_{0}}\right)(f(x_{0})-f(x^{*})),\dfrac{(8L+4M)\Delta_{\mathrm{max}}}{(1-\alpha)\kappa}\right\}\equiv\beta,
    \label{eq:3.2.4}
    \end{equation}
    for $k=0,\ldots,T$, where $\kappa$ is the constant in \eqref{sdcdf}.
\end{lemma}
\begin{proof}
    By the definition of $\beta$, (\ref{eq:3.2.4}) is true for $k=0$. Suppose that (\ref{eq:3.2.4}) is true for some $k\in\left\{0,\ldots,T-1\right\}$. Let us show that it is also true for $k+1$.
    \\[0.3cm]
    In the case where $\rho_k\geq \alpha$, by Step 3 of TRFD-S we have $\Delta_{k+1}\geq\Delta_{k}$. Since $f(x_{k+1})\leq f(x_{k})$, it follows that
    \begin{equation*}
    \left(\frac{1}{\Delta_{k+1}}\right)(f(x_{k+1})-f(x^{*}))\leq\left(\frac{1}{\Delta_{k}}\right)(f(x_{k})-f(x^{*}))\leq\beta,
    \end{equation*}
    where the last inequality is the induction assumption. Therefore, (\ref{eq:3.2.4}) holds for $k+1$ in this case. In the case where $\rho_k< \alpha$, by Step 4 of TRFD-S we have 
    \begin{equation}
          \Delta_{k+1}=\frac{1}{2}\Delta_{k}.
        \label{eq:3.2.5}
    \end{equation}
    In addition, in view of Lemma \ref{lem:3.2} and A3, we must have
    \begin{equation}
        \Delta_{k}>\dfrac{(1-\alpha)\kappa\eta_{\Delta_\mathrm{max}}(x_{k})}{2L+M},
        \label{eq:3.2.6}
    \end{equation}
    since otherwise, by Lemma \ref{lem:3.2}, we would have $\rho_k\geq \alpha$, contradicting our assumption that $\rho_k<\alpha$. Notice that (\ref{eq:3.2.6}) is equivalent to
    \begin{equation}\label{delta_eta}
     \left(\frac{1}{\Delta_{k}}\right)\eta_{\Delta_{\mathrm{max}}}(x_{k})<\dfrac{2L+M}{(1-\alpha)\kappa}.   
    \end{equation}
    Finally, it follows from (\ref{eq:3.2.5}), Lemma \ref{lem:3.2.2} and (\ref{delta_eta}) that
    \begin{eqnarray*}
        \left(\frac{1}{\Delta_{k+1}}\right)(f(x_{k+1})-f(x^{*}))&=&\left(\frac{2}{\Delta_{k}}\right)(f(x_{k+1})-f(x^{*})) = \left(\frac{2}{\Delta_{k}}\right)(f(x_{k})-f(x^{*}))\\
        &< &\dfrac{4\Delta_{\mathrm{max}}}{\Delta_{k}}\eta_{\Delta_{\mathrm{max}}}(x_{k}) < 4\Delta_{\mathrm{max}}\dfrac{2L+M}{(1-\alpha)\kappa}\\
        &\leq &\beta,
    \end{eqnarray*}
that is, (\ref{eq:3.2.4}) also holds for $k+1$ in this case, which concludes the proof.
\end{proof}

\noindent Let 
\begin{equation}
T_{f}(\epsilon)=\inf\left\{k\in\mathbb{N}\,:\,f(x_{k})-f(x^{*})\leq\Delta_{\mathrm{max}}\left(\frac{L}{\sigma}\epsilon\right)\right\}
\label{eq:3.2.7}
\end{equation}
be the first iteration index reaching a $\Delta_{\mathrm{max}}\left(\frac{L}{\sigma}\epsilon\right)$-approximate solution of (\ref{eq:first}) in $\Omega$, if it exists. Our goal is to establish a finite upper bound for $T_{f}(\epsilon)$. In this context, the lemma below provides an upper bound on $\left|\mathcal{S}_{T_{f}(\epsilon)-1}\right|$.
\begin{lemma}
    \label{lem:3.2.4}
    Suppose that A1-A3, A5 and A6 hold, and assume that $T_{f}(\epsilon)\geq 2$. If $\Delta_{\mathrm{max}}\geq D_{0}$, then
    \begin{equation}
    \left|\mathcal{S}_{T_f(\epsilon)-1}\right| \leq 1+\dfrac{2\beta}{\alpha\kappa}\left(\frac{L}{\sigma}\epsilon\right)^{-1},
    \label{eq:3.2.8}
    \end{equation}
    where $\beta$ is defined in \eqref{eq:3.2.4} and $\kappa$ is the constant in \eqref{sdcdf}.
\end{lemma}
\begin{proof}
    Let $k\in\mathcal{S}_{T_f(\epsilon)-2}$. By \eqref{ratio_generalset}, (\ref{sdcdf}), Lemmas \ref{lem:3.2.2}, \ref{lem:3.2.3} and A3, we have
    \begin{eqnarray}
    f(x_{k})-f(x_{k+1})&\geq &\alpha\kappa\eta_{\Delta_{\mathrm{max}}}(x_k) \min\left\{\Delta_k, \frac{\eta_{\Delta_\mathrm{max}}(x_k)}{\|H_k\|}\right\} \nonumber\\
    &\geq &\alpha\kappa\dfrac{f(x_{k})-f(x^{*})}{2\Delta_{\mathrm{max}}}\min\left\{\frac{f(x_{k})-f(x^{*})}{\beta}, \dfrac{f(x_{k})-f(x^{*})}{2\Delta_\mathrm{max}M}\right\}\nonumber\\
    &= &\dfrac{\alpha\kappa(f(x_{k})-f(x^{*}))^2}{2\Delta_\mathrm{max}}\min\left\{\frac{1}{\beta}, \dfrac{1}{2\Delta_\mathrm{max}M}\right\}.\nonumber
    \end{eqnarray}
    By the definition of $\beta$ in \eqref{eq:3.2.4}, since $\alpha \in (0,1)$, $\kappa \in (0,1)$ and $L > 0$, we get
    \begin{equation}\label{eq:3.2.9}
        f(x_{k})-f(x_{k+1}) \geq \dfrac{\alpha\kappa(f(x_k)-f(x^*))^2}{2\Delta_{\mathrm{max}}\beta}.
    \end{equation}
    Denoting $\delta_{k}=f(x_{k})-f(x^{*})$, (\ref{eq:3.2.9}) becomes
    \begin{equation*}
    \delta_{k}-\delta_{k+1}\geq\dfrac{\alpha\kappa}{2\Delta_\mathrm{max}\beta}\delta_{k}^{2}.
    \end{equation*}
    Consequently,
    \begin{equation}
    \dfrac{1}{\delta_{k+1}}-\dfrac{1}{\delta_{k}}=\dfrac{\delta_{k}-\delta_{k+1}}{\delta_{k}\delta_{k+1}}\geq\dfrac{\frac{\alpha\kappa}{2\Delta_\mathrm{max}\beta}\delta_{k}^{2}}{\delta_{k}^{2}}=\dfrac{\alpha\kappa}{2\Delta_\mathrm{max}\beta},\quad\text{when}\,\,k\in\mathcal{S}_{T_f(\epsilon)-2}.
    \label{eq:3.2.10}    
    \end{equation}
    Since $\delta_{k+1}=\delta_{k}$ for any $k\notin \mathcal{S}_{T_f(\epsilon)-2}$, it follows from (\ref{eq:3.2.10}) that
    \begin{eqnarray*}
        \dfrac{1}{\delta_{T_f(\epsilon)-1}}-\dfrac{1}{\delta_{0}}&=&\sum_{k=0}^{T_f(\epsilon)-2}\dfrac{1}{\delta_{k+1}}-\dfrac{1}{\delta_{k}}=\sum_{k\in\mathcal{S}_{T_f(\epsilon)-2}}\dfrac{1}{\delta_{k+1}}-\dfrac{1}{\delta_{k}}
        \geq \left|\mathcal{S}_{T_f(\epsilon)-2}\right|\dfrac{\alpha\kappa}{2\Delta_\mathrm{max}\beta}.
    \end{eqnarray*}
    Therefore, as $\delta_0>0$, we have
    \begin{equation*}
        \Delta_{\mathrm{max}}\left(\frac{L}{\sigma}\epsilon\right)<f(x_{T_f(\epsilon)-1})-f(x^{*})=\delta_{T_f(\epsilon)-1}\leq\dfrac{2\Delta_\mathrm{max}\beta}{\alpha\kappa\left|\mathcal{S}_{T_f(\epsilon)-2}\right|},
    \end{equation*}
    which implies
    \begin{equation*}
    \left|\mathcal{S}_{T_f(\epsilon)-1}\right|\leq 1 + \left|\mathcal{S}_{T_f(\epsilon)-2}\right| < 1+\dfrac{2\beta}{\alpha\kappa}\left(\frac{L}{\sigma}\epsilon\right)^{-1},
    \end{equation*}
    that is, \eqref{eq:3.2.8} is true.
\end{proof}

\noindent The next lemma establishes the relationship between $T_{f}(\epsilon)$ and $T_{g}(\epsilon)$.

\begin{lemma}\label{lem:tf}
    Suppose that A1, A2, A5 and A6 hold, and let $T_f(\epsilon)$ and $T_{g}(\epsilon)$ be defined by \eqref{eq:3.2.7} and \eqref{eq:hitting}, respectively. If $\Delta_{\mathrm{max}} \geq D_0$, then $T_f(\epsilon) \leq T_{g}(\epsilon)$.
\end{lemma}
\begin{proof}
    Suppose by contradiction that $T_f(\epsilon) > T_g(\epsilon)$. Then, by $\Delta_{\mathrm{max}} \geq D_0$, Lemma \ref{lem:3.2.1} and the definition of $T_g(\epsilon)$, we would have the contradiction
    \begin{equation*}
        \left(\frac{L}{\sigma}\epsilon\right) < \frac{f(x_{T_g(\epsilon)})-f(x^*)}{\Delta_{\mathrm{max}}} \leq \psi_{\Delta_{\mathrm{max}}}(x_{T_g(\epsilon)}) \leq \left(\frac{L}{\sigma}\epsilon\right).
    \end{equation*}
    So, we conclude that $T_f(\epsilon) \leq T_{g}(\epsilon)$.
\end{proof}
\noindent The following theorem gives an upper bound on the number of iterations required by TRFD-S to reach a $\Delta_{\mathrm{max}}\left(\frac{L}{\sigma}\epsilon\right)$-approximate solution of (\ref{eq:first}) in $\Omega$, when $f$ is a convex function.
\begin{theorem}
    Suppose that A1-A3, A5 and A6 hold, and let $T_f(\epsilon)$ be defined by \eqref{eq:3.2.7}. If $\Delta_{\mathrm{max}} \geq D_0$, then
    \begin{align}
        T_f(\epsilon) \; \leq \; \frac{4\beta}{\alpha\kappa}\left(\frac{L}{\sigma}\epsilon\right)^{-1} +\; \log_{2}\left(\frac{(8L+4M)\Delta_{0}}{(1-\alpha)\kappa}\left(\frac{L}{\sigma}\epsilon\right)^{-1}\right) + 2,
        \label{eq:3.14.1}
    \end{align}
    where $\beta$ is defined in \eqref{eq:3.2.4} and $\kappa$ is the constant in \eqref{sdcdf}.
    \label{thm:convex}
\end{theorem}
\begin{proof}
    If $T_f(\epsilon) \leq 1$, then \eqref{eq:3.14.1} is true. Let us assume that $T_f(\epsilon) \geq 2$. Similarly as in \eqref{eq:motivation}, we have
    \begin{equation}
        T_{f}(\epsilon) = \left|\mathcal{S}_{T_{f}(\epsilon)-1}\right| + \left|\mathcal{U}_{T_{f}(\epsilon)-1}^{(1)} \cup \mathcal{U}_{T_{f}(\epsilon)-1}^{(2)}\right|.\label{eq:3.14.2}
    \end{equation}
    By Lemma \ref{lem:tf}, we have $T_f(\epsilon) \leq T_g(\epsilon)$. Thus, by considering $T=T_f(\epsilon)$ in Lemma \ref{lem:3.5}, it follows that
\begin{equation}
\left|\mathcal{U}_{T_f(\epsilon)-1}^{(1)} \cup \mathcal{U}_{T_f(\epsilon)-1}^{(2)}\right| \leq \log_{2}\left(\frac{(8L+4M)\Delta_{0}}{(1-\alpha)\kappa}\left(\frac{L}{\sigma}\epsilon\right)^{-1}\right) + \left|\mathcal{S}_{T_f(\epsilon)-1}\right|.\label{eq:3.14.3}
\end{equation}
Then, by combining \eqref{eq:3.14.2}, Lemma \ref{lem:3.2.4} and \eqref{eq:3.14.3}, we conclude that \eqref{eq:3.14.1} is true.
\end{proof}

\noindent Since each iteration of TRFD-S requires at most $(n+1)$ function evaluations, from Theorem \ref{thm:convex} we obtain the following upper bound on the number of function evaluations required by TRFD-S to find a $\Delta_{\mathrm{max}}\left(\frac{L}{\sigma}\epsilon\right)$-approximate solution of (\ref{eq:first}) in $\Omega$, when $f$ is a convex function.
\vspace{0.2cm}
\begin{corollary}\label{coro:3.15}
    Suppose that A1-A3, A5 and A6 hold, and let $FE_{T_f(\epsilon)}$ be the number of function evaluations executed by TRFD-S up to the $(T_f(\epsilon)-1)$-$st$ iteration. If $\Delta_{\mathrm{max}} \geq D_0$, then
    \begin{eqnarray}\label{cor:3.19}
FE_{T_f(\epsilon)} \; \leq \; (n+1)\left[\frac{4\beta}{\alpha\kappa}\left(\frac{L}{\sigma}\epsilon\right)^{-1} +\; \log_{2}\left(\frac{(8L+4M)\Delta_{0}}{(1-\alpha)\kappa}\left(\frac{L}{\sigma}\epsilon\right)^{-1}\right)+2\right].
\end{eqnarray}
\end{corollary}
\noindent In view of \eqref{cor:3.19} and the definition of $\beta$ in \eqref{eq:3.2.4}, TRFD-S needs no more than $$\mathcal{O}\left(n\left(\frac{\sigma}{L}\right)L\Delta_{\mathrm{max}}\epsilon^{-1}\right)$$ function evaluations to find $x_k\in \Omega$ such that $f(x_k)-f(x^*) \leq \Delta_{\mathrm{max}}\left(\frac{L}{\sigma}\epsilon\right).$ Thus, given $\epsilon_f > 0$, if we use TRFD-S with $\epsilon = \epsilon_f/\Delta_{\mathrm{max}}$, then it will need no more than
\begin{equation}\label{cor_c}
    \mathcal{O}\left(n\left(\frac{\sigma}{L}\right)L\Delta_{\mathrm{max}}^2\epsilon_f^{-1}\right)
\end{equation}
function evaluations to find $x_k\in \Omega$ such that $f(x_k)-f(x^*) \leq \frac{L}{\sigma}\epsilon_f.$ So, in the case where $\sigma=L$, we get a worst-case evaluation complexity of
\begin{equation}\label{sigma_L_convex}
    \mathcal{O}\left(n L\Delta_{\mathrm{max}}^2\epsilon_f^{-1}\right)
\end{equation}
to satisfy $f(x_{k}) - f(x^*) \leq \epsilon_f$. Otherwise, when $\sigma \neq L$, Table \ref{tab:2} gives the different scenarios depending on the value of $\sigma$.
\begin{table}[h]
\centering
\begin{tabular}{|c|c|c|}
\hline
\textbf{Value of $\sigma$} & \textbf{Impact on \eqref{cor_c}} & \makecell{\textbf{Impact on the target} \\ \textbf{accuracy $\left(\frac{L}{\sigma}\epsilon_f\right)$}} \\ \hline
$\sigma<L$ & \makecell{\eqref{cor_c} lower than \eqref{sigma_L_convex} \\ by a linear factor $\left(\frac{\sigma}{L}\right)$} & \makecell{Accuracy weaker than $\epsilon_f$ \\ by a factor $\left(\tfrac{L}{\sigma}\right)$} \\ \hline
$\sigma>L$ & \makecell{\eqref{cor_c} larger than \eqref{sigma_L_convex} \\ by a linear factor $\left(\frac{\sigma}{L}\right)$} & \makecell{Accuracy stricter than $\epsilon_f$ \\ by a factor $\left(\frac{L}{\sigma}\right)$} \\ \hline
\end{tabular}
\caption{Impacts of the user-defined parameter $\sigma$ for convex problems}
\label{tab:2}
\end{table}
\subsection{Worst-Case Complexity Bound for P-L functions}

For the case where $f$ is a Polyak-Lojasiewicz (P-L) function \cite{polyak}, we will assume that the feasible set is unconstrained, i.e., $\Omega=\mathbb{R}^n$. Therefore, given $x_k\in \mathbb{R}^n$, the stationarity measure $\psi_{\Delta_{\mathrm{max}}}(x_k)$ reduces to $\|\nabla f(x_k)\|$, while the approximate stationarity measure $\eta_{\Delta_{\mathrm{max}}}(x_k)$ reduces to $\|g_k\|$.
\vspace{2mm}
\\
\noindent Now, let us consider the following assumption:
\vspace{2mm}
\begin{mdframed}
\textbf{A7.} $f$ is a P-L function, i.e., it has a global minimizer $x^* \in \mathbb{R}^n$ and
\begin{equation}\label{str_PL}
    \|\nabla f(x)\|^2 \geq 2\mu\left(f(x)-f(x^*)\right), \quad\forall x \in \mathbb{R}^n,
\end{equation}
\textcolor{black}{for some $\mu > 0$}.
\end{mdframed}
\vspace{2mm}
\noindent The following lemma relates the approximate stationarity measure $\|g_k\|$ with the functional residual.
\begin{lemma}\label{f_mu_g}
    Suppose that A2 and A7 hold, and assume that $\Omega=\mathbb{R}^n$. Moreover, let $x_k$ be generated by TRFD-S. If
    \begin{equation}\label{new_tol}
        f(x_k) - f(x^*) > \frac{1}{\mu}\left(\frac{L}{\sigma}\epsilon\right)^2,
    \end{equation}
    then
    \begin{equation}\label{sqrt_fc_g}
         \|g_k\|\geq \sqrt{\frac{\mu}{2}}(f(x_{k})-f(x^{*}))^{1/2}.
    \end{equation}
\end{lemma}
\begin{proof}
By \eqref{new_tol} and A7, we have
    \begin{equation*}
        \left(\frac{L}{\sigma}\epsilon\right) < \sqrt{\mu}(f(x_{k})-f(x^{*}))^{1/2} < \sqrt{2\mu}(f(x_{k})-f(x^{*}))^{1/2} \leq \|\nabla f(x_k)\|.
    \end{equation*}
Then, by Lemma \ref{eta_k_0}, we have that \eqref{eta_psi} holds. Therefore, by combining \eqref{eta_psi} and \eqref{str_PL}, we conclude that \eqref{sqrt_fc_g} is true.    
\end{proof}
\noindent The next lemma provides an upper bound on $\frac{\left(f(x_k)-f(x^*)\right)^{1/2}}{\Delta_k}$.
\begin{lemma}\label{delta_f_gamma}
    Suppose that A2, A3 and A7 hold, and assume that $\Omega=\mathbb{R}^n$. Moreover, given $T \geq 1$, let $\left\{x_k\right\}_{k=0}^{T}$ and $\left\{\Delta_k\right\}_{k=0}^{T}$ be generated by TRFD-S. If
    \begin{equation*}
        f(x_k) - f(x^*) > \frac{1}{\mu}\left(\frac{L}{\sigma}\epsilon\right)^2, \quad \text{for} \quad k=0,...,T-1,
    \end{equation*}
    then
    \small
    \begin{equation}
    \left(\frac{1}{\Delta_{k}}\right)(f(x_{k})-f(x^{*}))^{1/2}\leq\max\left\{\left(\frac{1}{\Delta_{0}}\right)(f(x_{0})-f(x^{*}))^{1/2},\sqrt{\frac{2}{\mu}}\dfrac{(4L+2M)}{(1-\alpha)\kappa}\right\}\equiv\gamma,
    \label{eq:3.3.1}
    \end{equation}
    \normalsize
    for $k=0,\ldots,T$, where $\kappa$ is the constant in \eqref{sdcdf}.
\end{lemma}
\begin{proof}
    Let us work through an induction argument. For $k=0$, \eqref{eq:3.3.1} clearly holds. Now, let us assume that \eqref{eq:3.3.1} is true for some $k \in \{0,...,T-1\}$. In the case where $\rho_k \geq \alpha$, similarly as in Lemma \ref{lem:3.2.3}, we get
    \begin{equation*}
        \left(\frac{1}{\Delta_{k+1}}\right)(f(x_{k+1})-f(x^{*}))^{1/2} \leq \left(\frac{1}{\Delta_{k}}\right)(f(x_{k})-f(x^{*}))^{1/2} \leq \gamma.
    \end{equation*}
    So, \eqref{eq:3.3.1} holds in this case. Let us now consider the case where $\rho_k < \alpha$. Similarly as in Lemma \ref{lem:3.2.3}, we have
    \begin{equation}\label{str_g}
        \frac{\|g_k\|}{\Delta_k} < \frac{2L+M}{(1-\alpha)\kappa}.
    \end{equation}
    Then, since $\Delta_{k+1}=\frac{1}{2}\Delta_k$ and $f(x_{k+1})=f(x_k)$, by combining Lemma \ref{f_mu_g} and \eqref{str_g}, we get
    \begin{eqnarray*}
        \left(\frac{1}{\Delta_{k+1}}\right)(f(x_{k+1})-f(x^{*}))^{1/2} &=& \left(\frac{2}{\Delta_{k}}\right)(f(x_{k})-f(x^{*}))^{1/2} \\
        &\leq & 2\sqrt{\frac{2}{\mu}}\frac{\|g_k\|}{\Delta_k} < \sqrt{\frac{2}{\mu}}\frac{4L+2M}{(1-\alpha)\kappa} \\
        &\leq& \gamma.
    \end{eqnarray*}
    So, \eqref{eq:3.3.1} is also true in this case, which concludes the proof.
\end{proof}
\vspace{2mm}
\noindent Now, let 
\begin{equation}
T_{PL}(\epsilon)=\inf\left\{k\in\mathbb{N}\,:\,f(x_{k})-f(x^{*})\leq\frac{1}{\mu}\left(\frac{L}{\sigma}\epsilon\right)^2\right\}
\label{eq:new_tol}
\end{equation}
be the first iteration index reaching a $\frac{1}{\mu}\left(\frac{L}{\sigma}\epsilon\right)^2$-approximate solution of (\ref{eq:first}) in $\mathbb{R}^n$, if it exists. Our goal is to establish a finite upper bound for $T_{PL}(\epsilon)$. In this context, the following lemma gives an upper bound for $\left|S_{T_{PL}(\epsilon)-1}\right|$.
\begin{lemma}\label{lem:str_S}
    Suppose that A2, A3 and A7 hold, and assume that $\Omega=\mathbb{R}^n$. If $T_{PL}(\epsilon) \geq 2$, then
    \begin{equation}\label{str_st}
        \left|S_{T_{PL}(\epsilon)-1}\right| \leq 1+ \frac{\log\left(\mu(f(x_0)-f(x^*))\left(\frac{L}{\sigma}\epsilon\right)^{-2}\right)}{\left|\log\left(1-\frac{\alpha\kappa(4L+2M)}{\gamma^2}\right)\right|},
    \end{equation}
    where $\gamma$ is defined in \eqref{eq:3.3.1} and $\kappa$ is the constant in \eqref{sdcdf}.
\end{lemma}
\begin{proof}
    Let $k\in\mathcal{S}_{T_{PL}(\epsilon)-2}$. By \eqref{ratio_generalset}, (\ref{sdcdf}), Lemmas \ref{f_mu_g}, \ref{delta_f_gamma} and A3, we have
    \begin{eqnarray}
    & &f(x_{k})-f(x_{k+1})\nonumber\\
    &\geq &\alpha\kappa\|g_k\| \min\left\{\Delta_k, \frac{\|g_k\|}{\|H_k\|}\right\} \nonumber\\
    &\geq &\alpha\kappa\sqrt{\dfrac{\mu}{2}}\left(f(x_{k})-f(x^{*})\right)^{1/2}\min\left\{\frac{\left(f(x_{k})-f(x^{*})\right)^{1/2}}{\gamma}, \sqrt{\dfrac{\mu}{2}}\dfrac{\left(f(x_{k})-f(x^{*})\right)^{1/2}}{M}\right\}\nonumber\\
    &= &\alpha\kappa\sqrt{\dfrac{\mu}{2}}\left(f(x_{k})-f(x^{*})\right)\min\left\{\frac{1}{\gamma}, \sqrt{\dfrac{\mu}{2}}\dfrac{1}{M}\right\}.
    \label{eq:dimanche}
    \end{eqnarray}
    By the definition of $\gamma$ in \eqref{eq:3.3.1}, we have $\gamma \geq \sqrt{\frac{2}{\mu}}\frac{4L+2M}{(1-\alpha)\kappa}$. So, since $\alpha \in (0,1)$ and $\kappa \in (0,1)$, we get
    \begin{equation}\label{gamma_prop_2}
        \sqrt{\frac{\mu}{2}} \geq \frac{4L+2M}{\gamma},
    \end{equation}
    which implies
    \begin{equation}\label{gamma_prop_1}
        \frac{1}{\gamma} \leq \sqrt{\frac{\mu}{2}}\frac{1}{M}.
    \end{equation}
    Therefore, by (\ref{eq:dimanche}), \eqref{gamma_prop_2} and \eqref{gamma_prop_1}, it follows that
    \begin{align}
        f(x_k) - f(x_{k+1}) &\geq \frac{\alpha\kappa(4L+2M)}{\gamma^2}\left(f(x_{k})-f(x^{*})\right). \label{str_gamma_sq}
    \end{align}
    Denoting $\delta_k = f(x_k) - f(x^*)$, \eqref{str_gamma_sq} becomes
    \begin{equation*}
        \delta_k - \delta_{k+1} \geq \frac{\alpha\kappa(4L+2M)}{\gamma^2}\delta_k,
    \end{equation*}
    which gives
    \begin{equation*}
        \delta_{k+1} \leq \left(1-\frac{\alpha\kappa(4L+2M)}{\gamma^2}\right)\delta_k,\quad\text{when}\,\,k\in\mathcal{S}_{T_{PL}(\epsilon)-2},
    \end{equation*}
    where $\frac{\alpha\kappa(4L+2M)}{\gamma^2}<1$ by the definition of $\gamma$ in \eqref{eq:3.3.1} and by $\mu \leq L$. Then, since $\delta_{k+1} = \delta_k$ when $k\notin \mathcal{S}_{T_{PL}(\epsilon)-2}$, we have
    \begin{align*}
    \frac{1}{\mu}\left(\frac{L}{\sigma}\epsilon\right)^2 &< f(x_{T_{PL}(\epsilon)-1}) - f(x^*) = \delta_{T_{PL}(\epsilon)-1} \leq \prod_{k\in \mathcal{S}_{T_{PL}(\epsilon)-2}}\left(1-\frac{\alpha\kappa(4L+2M)}{\gamma^2}\right) \delta_0\\ &= \left(1-\frac{\alpha\kappa(4L+2M)}{\gamma^2}\right)^{\left|\mathcal{S}_{T_{PL}(\epsilon)-2}\right|}(f(x_0)-f(x^*)),
    \end{align*}
    which is equivalent to
    \begin{equation*}
        \left(1-\frac{\alpha\kappa(4L+2M)}{\gamma^2}\right)^{\left|\mathcal{S}_{T_{PL}(\epsilon)-2}\right|} > \frac{1}{\mu(f(x_0)-f(x^*))}\left(\frac{L}{\sigma}\epsilon\right)^2. 
    \end{equation*}
    Then, taking the logarithm on both sides,
    \begin{equation*}
    \left|\mathcal{S}_{T_{PL}(\epsilon)-2}\right| \log\left(1-\frac{\alpha\kappa(4L+2M)}{\gamma^2}\right) > \log\left(\frac{1}{\mu(f(x_0)-f(x^*))}\left(\frac{L}{\sigma}\epsilon\right)^2\right).
    \end{equation*}
    So,
    \begin{equation*}
        \left|\mathcal{S}_{T_{PL}(\epsilon)-1}\right|\leq 1+\left|\mathcal{S}_{T_{PL}(\epsilon)-2}\right| < 1+ \frac{\log\left(\mu(f(x_0)-f(x^*))\left(\frac{L}{\sigma}\epsilon\right)^{-2}\right)}{\left|\log\left(1-\frac{\alpha\kappa(4L+2M)}{\gamma^2}\right)\right|},
    \end{equation*}
    which shows that \eqref{str_st} is true.
\end{proof}
\noindent The next lemma shows that $T_{PL}(\epsilon) \leq T_g(\epsilon)$.
\begin{lemma}\label{TPL_Tg}
    Suppose that A2 and A7 hold, and assume that $\Omega = \mathbb{R}^n$. Moreover, let $T_{PL}(\epsilon)$ and $T_g(\epsilon)$ be defined by \eqref{eq:new_tol} and \eqref{eq:hitting}, respectively. Then, $T_{PL}(\epsilon) \leq T_g(\epsilon)$.
\end{lemma}
\begin{proof}
    Suppose by contradiction that $T_{PL}(\epsilon) > T_g(\epsilon)$. Then, by \eqref{str_PL} and the definition of $T_g(\epsilon)$, we would have
    \small
    \begin{equation*}
        \left(\frac{L}{\sigma}\epsilon\right) < \sqrt{\mu}(f(x_{T_{g}(\epsilon)})-f(x^{*}))^{1/2} < \sqrt{2\mu}(f(x_{T_{g}(\epsilon)})-f(x^{*}))^{1/2} \leq \|\nabla f(x_{T_g(\epsilon)})\| \leq \left(\frac{L}{\sigma}\epsilon\right),
    \end{equation*}
    \normalsize
    leading to a contradiction. So, we conclude that $T_{PL}(\epsilon) \leq T_g(\epsilon)$.
\end{proof}
\noindent The next theorem gives an upper bound on the number of iterations required by TRFD-S to reach a $\frac{1}{\mu}\left(\frac{L}{\sigma}\epsilon\right)^2$-approximate solution of (\ref{eq:first}) in $\mathbb{R}^n$, when $f$ is a P-L function.
\begin{theorem}
    Suppose that A2, A3 and A7 hold, and assume that $\Omega=\mathbb{R}^n$. Moreover, let $T_{PL}(\epsilon)$ be defined by \eqref{eq:new_tol}. Then
    \small
    \begin{align}
        T_{PL}(\epsilon) \; \leq \; \frac{2\log\left(\mu(f(x_0)-f(x^*))\left(\frac{L}{\sigma}\epsilon\right)^{-2}\right)}{\left|\log\left(1-\frac{\alpha\kappa(4L+2M)}{\gamma^2}\right)\right|} +\; \log_{2}\left(\frac{(8L+4M)\Delta_{0}}{(1-\alpha)\kappa}\left(\frac{L}{\sigma}\epsilon\right)^{-1}\right) + 2,
        \label{eq:3.19.1}
    \end{align}
    \normalsize
    where $\gamma$ is defined in \eqref{eq:3.3.1} and $\kappa$ is the constant in \eqref{sdcdf}.
    \label{thm:str_convex}
\end{theorem}
\begin{proof}
    If $T_{PL}(\epsilon) \leq 1$, then \eqref{eq:3.19.1} is true. Let us assume that $T_{PL}(\epsilon) \geq 2$. As in \eqref{eq:motivation}, we have
    \begin{equation}
        T_{PL}(\epsilon) = \left|\mathcal{S}_{T_{PL}(\epsilon)-1}\right| + \left|\mathcal{U}_{T_{PL}(\epsilon)-1}^{(1)} \cup \mathcal{U}_{T_{PL}(\epsilon)-1}^{(2)}\right|.\label{eq:3.14.2_str}
    \end{equation}
    Moreover, by Lemma \ref{TPL_Tg} we have $T_{PL}(\epsilon) \leq T_g(\epsilon)$. Thus, by considering $T=T_{PL}(\epsilon)$ in Lemma \ref{lem:3.5}, it follows that
    \begin{equation}
\left|\mathcal{U}_{T_{PL}(\epsilon)-1}^{(1)} \cup \mathcal{U}_{T_{PL}(\epsilon)-1}^{(2)}\right| \leq \log_{2}\left(\frac{(8L+4M)\Delta_{0}}{(1-\alpha)\kappa}\left(\frac{L}{\sigma}\epsilon\right)^{-1}\right) + \left|\mathcal{S}_{T_{PL}(\epsilon)-1}\right|\label{eq:3.14.3_str}
\end{equation}
    holds. Then, combining \eqref{eq:3.14.2_str}, Lemma \ref{lem:str_S} and \eqref{eq:3.14.3_str}, we conclude that \eqref{eq:3.19.1} is true.
\end{proof}
\noindent Since each iteration of TRFD-S requires at most $(n+1)$ function evaluations, from Theorem \ref{thm:str_convex} we obtain the following upper bound on the number of function evaluations required by TRFD-S to find a $\frac{1}{\mu}\left(\frac{L}{\sigma}\epsilon\right)^2$-approximate solution of (\ref{eq:first}) in $\mathbb{R}^n$, when $f$ is a P-L function.
\begin{corollary}\label{cor:3.20}
    Suppose that A2, A3 and A7 hold, and assume that $\Omega=\mathbb{R}^n$. Moreover, let $FE_{T_{PL}(\epsilon)}$ be the number of function evaluations executed by TRFD-S up to the $(T_{PL}(\epsilon)-1)$-$st$ iteration. Then
    \small
    \begin{eqnarray}
FE_{T_{PL}(\epsilon)} &\leq &(n+1)\left[\frac{2\log\left(\mu(f(x_0)-f(x^*))\left(\frac{L}{\sigma}\epsilon\right)^{-2}\right)}{\left|\log\left(1-\frac{\alpha\kappa(4L+2M)}{\gamma^2}\right)\right|} +\; \log_{2}\left(\frac{(8L+4M)\Delta_{0}}{(1-\alpha)\kappa}\left(\frac{L}{\sigma}\epsilon\right)^{-1}\right)+2\right]\nonumber\\
& &
\label{cor:str}
\end{eqnarray}
\normalsize
\end{corollary}
\noindent In view of \eqref{cor:str} and the definition of $\gamma$ in \eqref{eq:3.3.1}, TRFD-S requires at most
\begin{equation*}
\mathcal{O}\left(n\frac{L}{\mu}\log\left(\left(\frac{\sigma}{L}\right)^2\mu(f(x_0)-f(x^*))\epsilon^{-2}\right)\right)
\end{equation*}
function evaluations to find $x_k\in \mathbb{R}^n$ such that $f(x_k)-f(x^*) \leq \frac{1}{\mu}\left(\frac{L}{\sigma}\epsilon\right)^2.$ Thus, given $\epsilon_f>0$, if $\epsilon = \sqrt{\left(\frac{\sigma}{L}\right)\mu\epsilon_f}$, then TRFD-S requires no more than
\begin{equation}\label{cor:res}
\mathcal{O}\left(n\frac{L}{\mu}\log\left(\left(\frac{\sigma}{L}\right)(f(x_0)-f(x^*))\epsilon_f^{-1}\right)\right)
\end{equation}
function evaluations to find $x_k\in \mathbb{R}^n$ such that $f(x_k)-f(x^*)\leq \frac{L}{\sigma}\epsilon_f$. Therefore, when the user-defined parameter $\sigma$ equals $L$, we obtain a worst-case evaluation complexity of
\begin{equation}\label{sigma_L_str_convex}
    \mathcal{O}\left(n\frac{L}{\mu} \log\left(\left(f(x_0)-f(x^*)\right)\epsilon_f^{-1}\right)\right)
\end{equation}
to satisfy $f(x_{k}) - f(x^*) \leq \epsilon_f$. Otherwise, when $\sigma \neq L$, Table \ref{tab:3} summarizes the different impacts of $\sigma$.
\begin{table}[h]
\centering
\begin{tabular}{|c|c|c|}
\hline
\textbf{Value of $\sigma$} & \textbf{Impact on \eqref{cor:res}} & \makecell{\textbf{Impact on the target} \\ \textbf{accuracy $\left(\frac{L}{\sigma}\epsilon_f\right)$}} \\ \hline
$\sigma<L$ & \makecell{\eqref{cor:res} lower than \eqref{sigma_L_str_convex} \\ by an additive term $n\frac{L}{\mu}\log\left(\frac{\sigma}{L}\right)$} & \makecell{Accuracy weaker than $\epsilon_f$ \\ by a factor $\left(\frac{L}{\sigma}\right)$} \\ \hline
$\sigma>L$ & \makecell{\eqref{cor:res} larger than \eqref{sigma_L_str_convex} \\ by an additive term $n\frac{L}{\mu}\log\left(\frac{\sigma}{L}\right)$} & \makecell{Accuracy stricter than $\epsilon_f$ \\ by a factor $\left(\frac{L}{\sigma}\right)$} \\ \hline
\end{tabular}
\caption{Impacts of the user-defined parameter $\sigma$ for \textcolor{black}{P-L functions}}
\label{tab:3}
\end{table}
\subsection{Trust-Region Method for Unrelaxable Bound Constraints}\label{sec:unrelaxable}
In this subsection, we propose an adaptation of TRFD-S for unrelaxable bound constraints problems, i.e., the case where $\Omega = [\ell,u]$ with $\ell,u \in \mathbb{R}^n$ being lower and upper bounds on the variables, respectively, and where $f$ cannot be evaluated outside $\Omega$. \textcolor{black}{Such scenarios typically appear in parameter tuning}, where the parameters have a particular range of values for intrinsic reasons \cite{alarie}.

As it is, Step 1 of TRFD-S may require evaluating $f$ at points outside $\Omega$. For problems where this is not feasible, we modify the step as follows. For each component $i \in \{1, \dots, n\}$ of $g_k$, forward and backward finite-difference stepsizes, $\tau_{k,i}^F$ and $\tau_{k,i}^B$, are initially set to a default $\tau_k$. These stepsizes are reduced if necessary to ensure that $x_k + \tau_{k,i}^F e_i$ and $x_k - \tau_{k,i}^B e_i$ remain in $\Omega$:
\[
\tau_{k,i}^F = \min\{[u - x_k]_{i}, \tau_k\}, \quad \tau_{k,i}^B = \min\{[x_k - \ell]_{i}, \tau_k\}.
\]
Since one of these stepsizes might become too small—or even zero—we take, for each $i$, the larger of $\tau_{k,i}^F$ and $\tau_{k,i}^B$ as the effective stepsize to avoid numerical errors. Figure \ref{fig:box} illustrates this procedure in a two-dimensional box.
\begin{figure}[h!]
    \centering
    \includegraphics[width=0.50\linewidth]{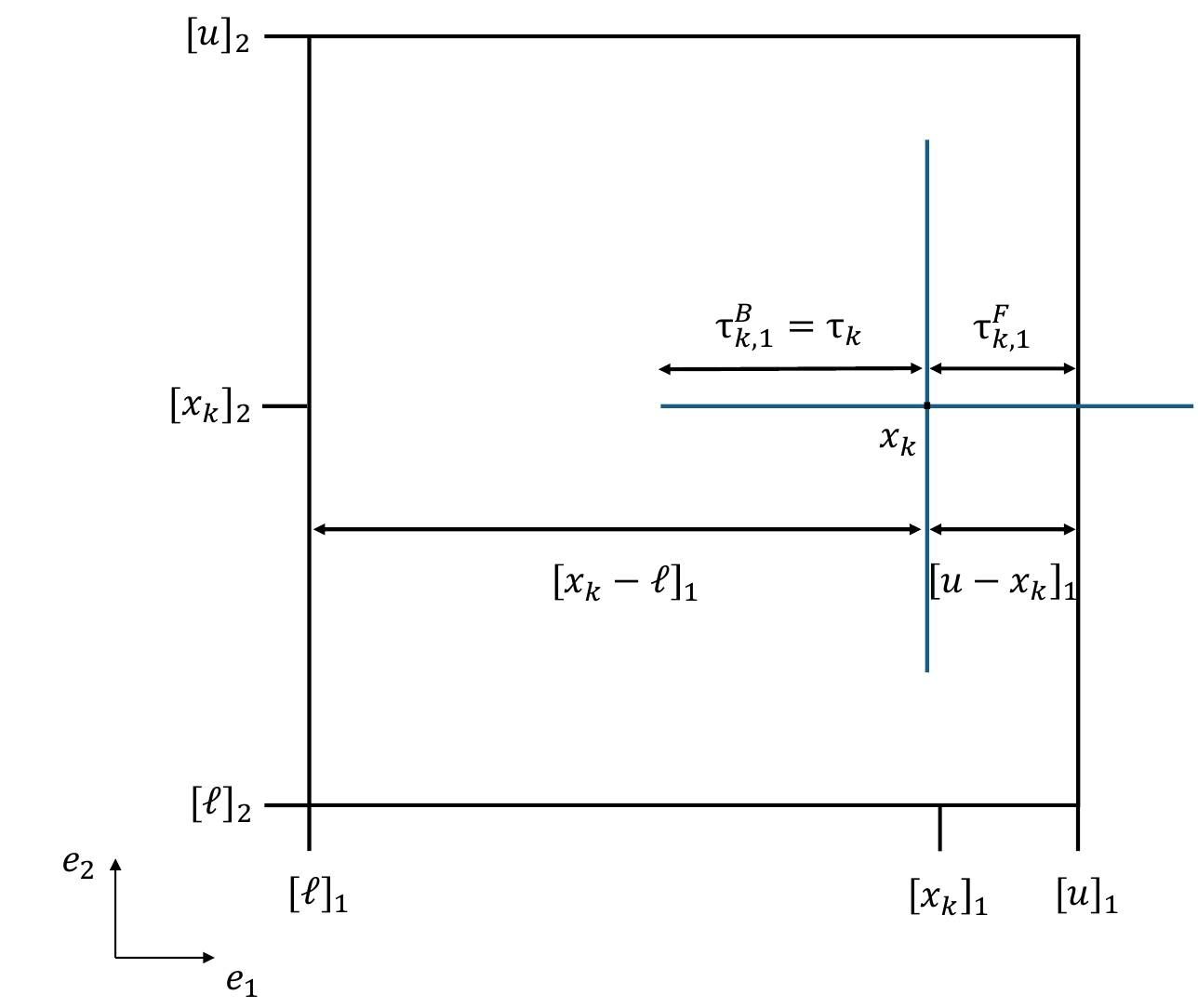}
    \caption{Illustration of Step 1 for unrelaxable bound constraints}
    \label{fig:box}
\end{figure}

Specifically, adapting TRFD-S (Algorithm 1) to handle unrelaxable bound constraints requires only a modification of Step 1; all other steps remain unchanged.

\begin{mdframed}
\noindent\textbf{Adaptation of Step 1 in TRFD-S} for unrelaxable bound constraints
\\[0.2cm]
\noindent\textbf{Step 1.} Let
\begin{equation*}
\tau_{k,i}^F = \min\{\left[u-x_k\right]_i,\tau_k\} \quad \text{and} \quad \tau_{k,i}^B = \min\{\left[x_k-\ell\right]_i,\tau_k\}, \quad i=1,...,n.   
\end{equation*}
Compute each component of $g_k \in \mathbb{R}^n$ as
\begin{gather*}
    \begin{aligned}
    \left[g_k\right]_i=
        \begin{cases}
            \frac{f(x_k+\tau_{k,i}^Fe_i)-f(x_k)}{\tau_{k,i}^F}, \quad \text{if} \, \tau_{k,i}^F \geq \tau_{k,i}^B, \\
            \frac{f(x_k)-f(x_k-\tau_{k,i}^Be_i)}{\tau_{k,i}^B}, \quad \text{otherwise}.
        \end{cases}
    \end{aligned}
\end{gather*}
Choose a nonzero symmetric matrix $H_k\in \mathbb{R}^{n\times n}$.
\end{mdframed}
\begin{remark}
By constructing the vector $g_k$ as above, we have that Lemma \ref{firstlem} remains true. Moreover, since the update rules of $\tau_k$ and $\Delta_k$ are unchanged, we have that $g_k$ still satisfies \eqref{fg_delta}. So, we conclude that the worst-case complexity bounds established in Corollaries \ref{cor:3.1} and \ref{coro:3.15} for general nonempty closed convex sets $\Omega$ are also true for TRFD-S with unrelaxable bound constraints.   
\end{remark}

\section{Numerical Experiments}\label{sec:4}
{\color{black}
To assess the numerical performance of TRFD-S, we conducted experiments by using MATLAB implementations. First, we considered unconstrained benchmark problems. Secondly, we considered unrelaxable bound constraints benchmark problems. Finally, we looked at the model fitting of a synthetic Predator-Pray dataset, in an unrelaxable bound constraints setting.

For unconstrained benchmark problems (see subsection \ref{sec:4.1}), we compared TRFD-S against NEWUOA \cite{powell,zhang_2023}, DFQRM \cite{grapiglia2} and an instance of TRFD \cite{grapiglia}, without and with additive noise. For unrelaxable bound constraints benchmark problems (see subsection \ref{sec:4.2}), we compared TRFD-S with BOBYQA \cite{bobyqa,zhang_2023}, NOMAD \cite{audet} and an instance of TRFD, without and with additive noise. Finally, for the model fitting problem (see subsection \ref{sec:app}), we compared TRFD-S with BOBYQA.
}

For each problem, a budget of 100 simplex gradients was allowed to each solver\footnote{One simplex gradient corresponds to $n+1$ function evaluations, with $n$ being the number of variables of the problem.}. In addition, our implementations of TRFD-S were equipped with the stopping criterion
\begin{equation*}
    \Delta_k \leq 10^{-13}.
\end{equation*}
All implementations were compared by using data profiles\footnote{The data profiles were generated using the code \texttt{data\_profile.m}, freely available at the website\\ \url{https://www.mcs.anl.gov/~more/dfo/}.} \cite{more2009benchmarking}, where a code $M$ is said to solve a problem with some \textit{Tolerance} when it reaches $x_M$ such that $$f(x_0)-f(x_{M}) \geq \left(1-\textit{\text{Tolerance}}\right)\left(f(x_0)-f(x_{Best})\right),$$ where $f(x_{Best})$ is the lowest function value found among all the methods, and \textit{Tolerance} $\in(0,1)$. All experiments were performed with MATLAB (R2023a) on a PC with microprocessor 13-th Gen Intel(R) Core(TM) i5-1345U 1.60 GHz and 32 GB of RAM memory.

\subsection{Unconstrained Benchmark Problems}\label{sec:4.1}
Here, we considered smooth unconstrained problems. We tested 134 functions $f:\mathbb{R}^n \to \mathbb{R}$ defined by the OPM collection \cite{gratton_opm}, for which $2\leq n \leq 110$, and where the initial points $x_0$ were provided by the collection.

\subsubsection{Performance of TRFD-S with other DFO methods}\label{perfo_unc}

The following codes were compared:
\\[0.15cm]
- \textbf{TRFD-S}: Implementation of TRFD-S, freely available on GitHub\footnote{\url{https://github.com/danadavar/TRFD-S}}, with initial parameters: $\epsilon=10^{-5}$, $\alpha=0.01$, $\Delta_0=\max\left\{1, \tau_0\sqrt{n}\right\}$, $\Delta_{\mathrm{max}}=\max\left\{1000, \Delta_0\right\}$ and $\sigma=\frac{\epsilon}{\sqrt{n}\sqrt{eps}}$, where $eps$ is the machine precision. \textcolor{black}{Such definition for $\sigma$ ensures that $\tau_0=\sqrt{eps}$, which is the stepsize widely used for finite differences \cite[Chapter 8]{NW}}. \textcolor{black}{Such setting ranges the values of $\sigma$ from 64 to 475 on the OPM collection.} The matrix $H_k$ is updated according to the BFGS rule\footnote{{\color{black}Preliminary numerical tests motivated this choice of update, as BFGS outperformed both the SR1 update and a scaled identity matrix using the Barzilai--Borwein step size.}}:
\begin{align*}
    H_{k+1} =
    \begin{cases}
        H_k + \frac{y_ky_k^T}{s_k^Ty_k} - \frac{H_ks_ks_k^TH_k}{s_k^TH_ks_k} \quad \text{if} \quad  \left|\langle s_k,y_k\rangle\right| > 0, \\
        H_k \quad \text{otherwise,}
    \end{cases}
\end{align*}
with $H_0 = I$, $s_k = x_{k+1}-x_k$ and $y_k = g_{k+1}-g_k$. The trust-region subproblem is solved via the method proposed in \cite{adachi}, which is implemented in the function \texttt{TRSgep.m} from the MANOPT toolbox, freely available on GitHub\footnote{\url{https://github.com/NicolasBoumal/manopt}}.
\\[0.15cm]
- \textbf{NEWUOA}: Implementation of Powell's method \cite{powell,zhang_2023}, freely available on GitHub\footnote{\url{https://github.com/libprima/prima}}. The initial parameters were not changed.
\\[0.15cm]
- \textbf{DFQRM}: Implementation of the quadratic regularization method described in Section 4 of \cite{grapiglia2}.
\\[0.15cm]
- \textbf{TRFD-2}: Implementation of TRFD \cite{grapiglia}, with $p=2$, $m=1$ and $h(z)=z,\,\forall z\in\mathbb{R}$, freely available on GitHub\footnote{\url{https://github.com/danadavar/TRFD}}. The threshold is set to $\alpha=0.01$, while the other parameters follow the same setup as Section 4 in \cite{grapiglia}.
\\[0.15cm]
Data profiles are presented in Figure \ref{fig:1}. As we can see, TRFD-S outperforms DFQRM and TRFD-2 for all presented tolerances, while exhibiting a competitive performance with NEWUOA. Notably, TRFD-S achieves better results than NEWUOA for tolerances $10^{-5}$ and $10^{-7}$.
\begin{figure}[h!]
    \centering
    {\includegraphics[width=0.33\textwidth]{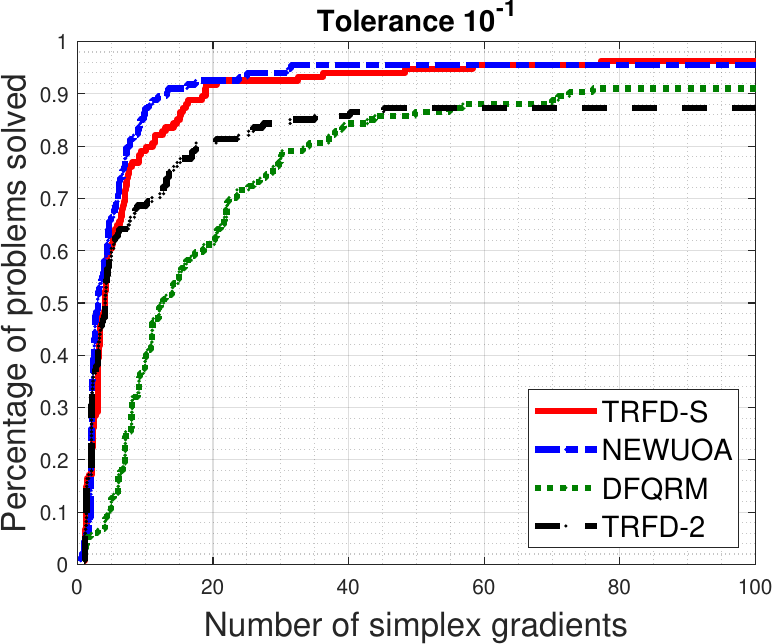}}
    \hspace{5mm}
    {\includegraphics[width=0.33\textwidth]{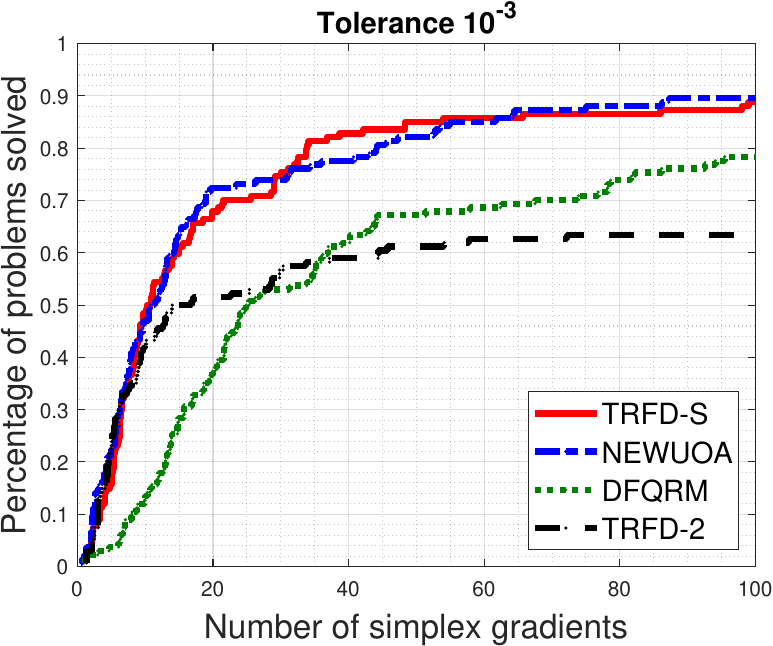}} 
    {\includegraphics[width=0.33\textwidth]{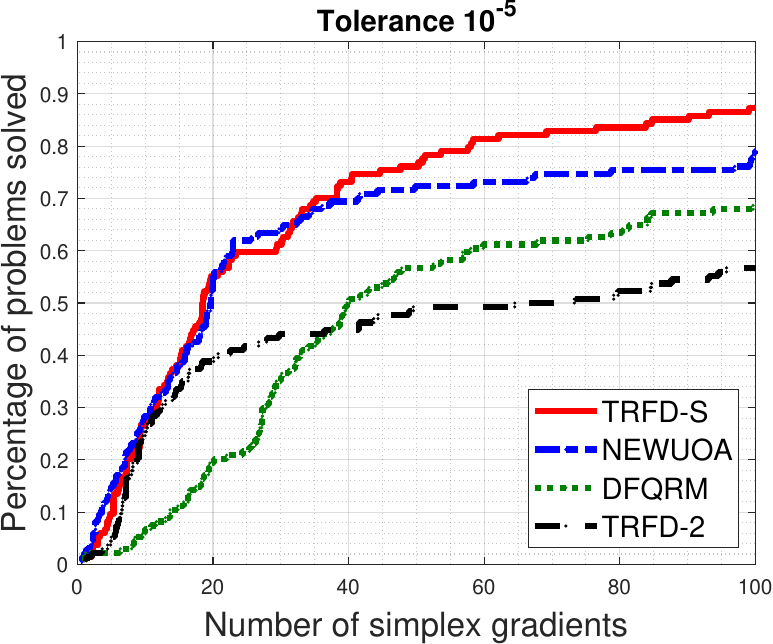}}
    \hspace{5mm}
    {\includegraphics[width=0.33\textwidth]{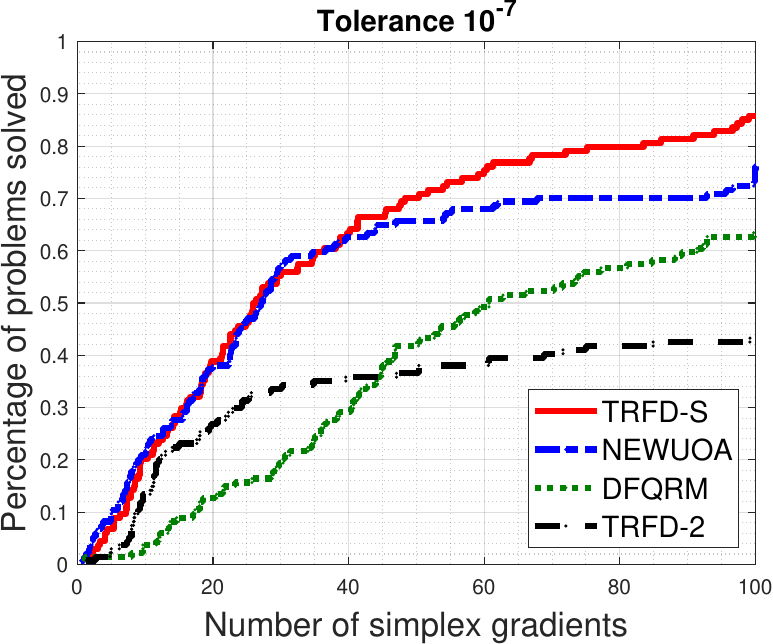}}
    \caption{Data profiles for smooth unconstrained problems.}
    \label{fig:1}
\end{figure}
\subsubsection{Performance of TRFD-S with other DFO methods with additive noise}\label{sec:5:noise}

Here, we considered additive noise by following the same setup as in \cite[Sections 2 and 4]{shi_nocedal}, i.e, the noise follows a uniform distribution with interval $[-\sqrt{3}, \sqrt{3}]$, and we have $s.d. \in \{10^{-1}, 10^{-3}, 10^{-5}, 10^{-7}\}$, where $s.d.$ is the standard deviation. Data profiles are shown on Figure \ref{fig:unc_noise}. As we can see, NEWUOA outperforms TRFD-S when $s.d.=10^{-1}$, while it shows a competitive behavior with TRFD-S on smaller levels of noise. \textcolor{black}{Qualitatively, this is similar to the results obtained in \cite{shi_nocedal} for a comparison between NEWUOA and another finite-difference based method.} 
\begin{figure}[h!]
    \centering
    {\includegraphics[width=0.33\textwidth]{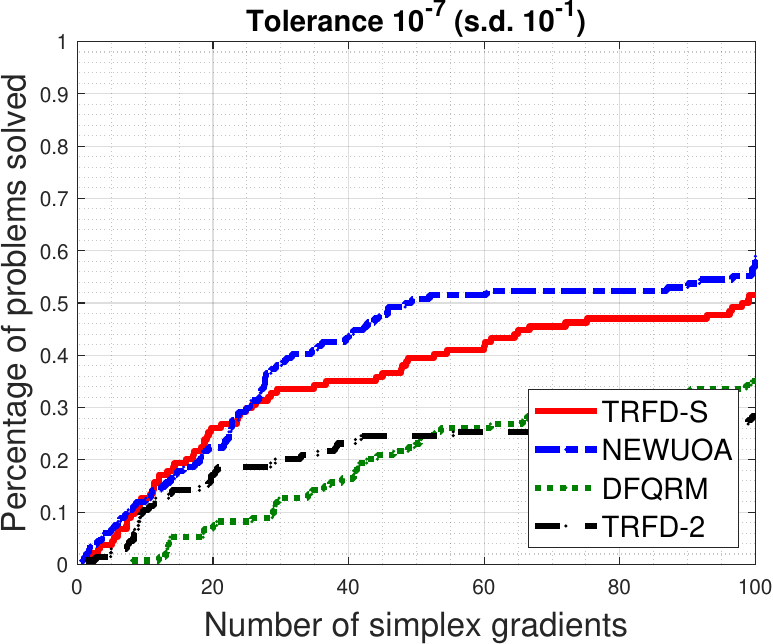}}
    \hspace{5mm}
    {\includegraphics[width=0.33\textwidth]{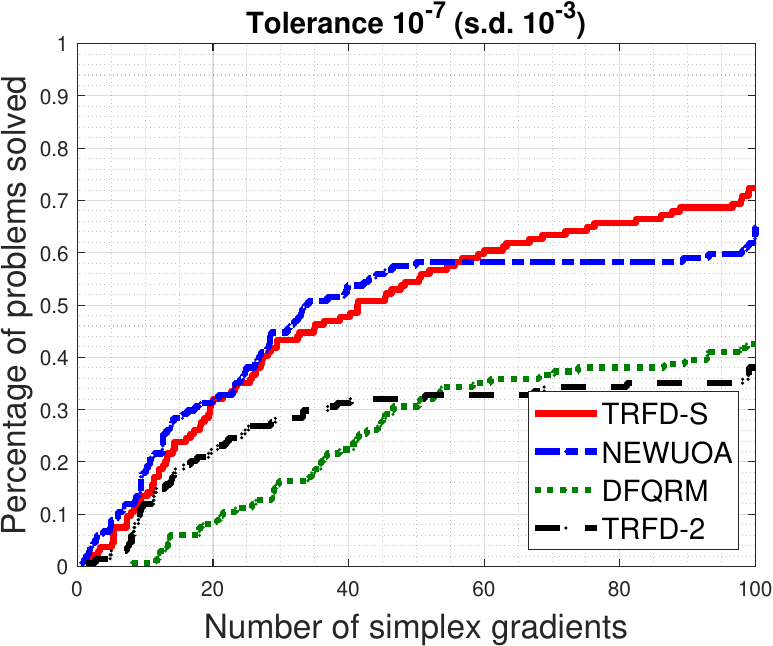}} 
    {\includegraphics[width=0.33\textwidth]{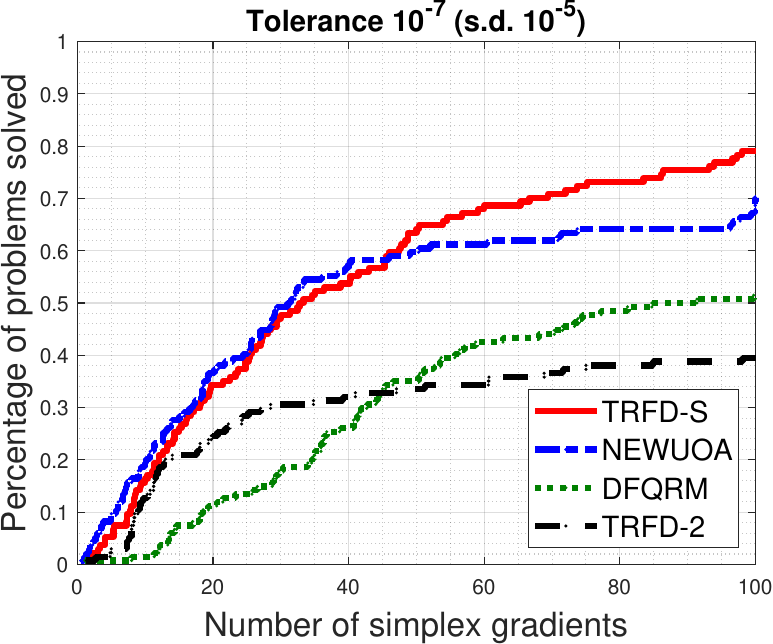}}
    \hspace{5mm}
    {\includegraphics[width=0.33\textwidth]{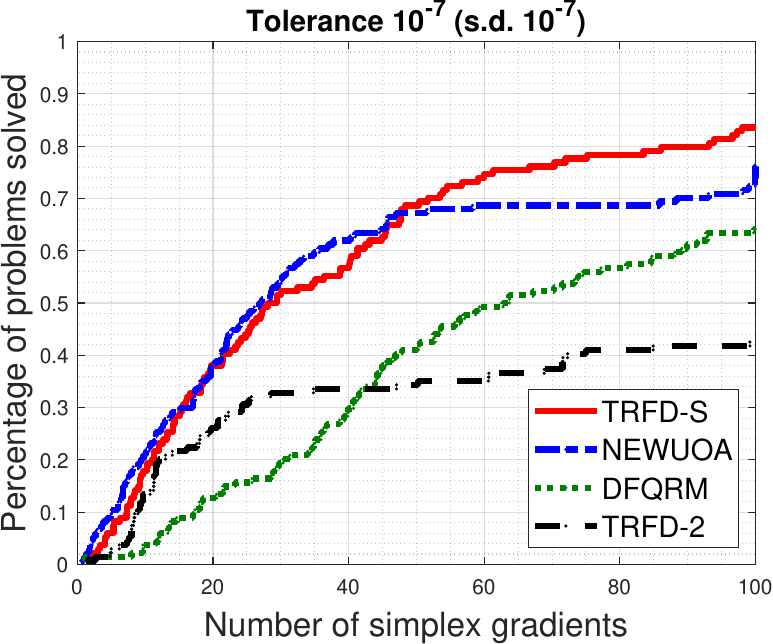}}
    \caption{Data profiles for smooth unconstrained problems with noise.}
    \label{fig:unc_noise}
\end{figure}
\subsection{Unrelaxable Bound Constraints Problems}\label{sec:4.2}
\noindent Here, we considered smooth problems of the form $$\min_{x\in \Omega} f(x) \equiv \left\|F(x)\right\|_{2}^2,$$ where $\Omega$ is defined by the unrelaxable bounds $\ell_i=0.1$ and $u_i=20$, for $i=1,...,n$, as set in Section 5 of \cite{hough}. We tested 53 functions $F:\mathbb{R}^n \to \mathbb{R}^m$ defined by the Moré-Wild collection \cite{more2009benchmarking}, for which $2\leq n \leq 12$ and $2 \leq m \leq 65$, and where the initial points $x_0$ were provided by the collection. In the case where $x_0$ violated the bound constraints, an orthogonal projection was applied to $\Omega$.

\subsubsection{Performance of TRFD-S with other DFO methods}\label{perf_dfo_bc} 

The following codes were compared:
\\[0.15cm]
- \textbf{TRFD-S}: Same setup as TRFD-S in subsection \ref{perfo_unc} but with the condition $\left\langle s_k,y_k\right\rangle > 0$ in the update of the Hessian approximation, {\color{black}which ensures to have a convex quadratic model. Such model is minimized over a convex set that is the intersection of a ball and a box}, by using the inner solver FISTA \cite{fista} supported by Dykstra's algorithm \cite{dykstra_1,dykstra_2}. The latter ensures that the points generated by FISTA $\left\{x_{k_{i}}\right\}_{i \geq 0}$ remain in the set $\mathcal{C}:=[\ell,u] \times \{s \in \mathbb{R}^n:\|s-x_{k_{i}}\|\leq \Delta_k\}$, leading to the iteration process $x_{k_{i+1}}=P_{\mathcal{C}}\left(x_{k_{i}}-\frac{1}{\tilde{L}}g_{k_{i}}\right)$, where $g_{k_{i}}$ is the gradient of $m_k$ at $x_{k_{i}}$, and $\tilde{L}=\|H_k\|$. The parameters were $100n^2$, $10^{-4}$ and $10^{-12}$ for the maximum number of iterations in FISTA, the stopping criterion in Dykstra's algorithm, and the stopping criterion for the distance between the iterates in FISTA, respectively\footnote{{\color{black}Preliminary numerical tests led to these choices of parameters.}}. 
\\[0.15cm]
- \textbf{BOBYQA}: Implementation of Powell's method \cite{bobyqa,zhang_2023}. The initial parameters were unchanged, while the option \texttt{honour\_x0} was set to \textit{true} to force the method not to move the position of the initial point.
\\[0.15cm]
- \textbf{NOMAD}: Implementation of the version 4 of NOMAD, proposed in \cite{audet}. The initial parameters were not changed.
\\[0.15cm]
- \textbf{TRFD-Inf}: Implementation of TRFD \cite{grapiglia}, freely available on GitHub\footnote{\url{https://github.com/danadavar/TRFD}}, with $p=+\infty$, $m=1$ and $h(z)=z,\,\forall z\in\mathbb{R}$. The threshold is the same as in subsection \ref{sec:4.1}, namely, $\alpha=0.01$, while the other parameters follow the same setup as in Section 4 of \cite{grapiglia}.
\\[0.15cm]
Data profiles are shown in Figure \ref{fig:4}. As shown, TRFD-S outperforms NOMAD and TRFD-Inf, while exhibiting a performance comparable to BOBYQA. In particular, TRFD-S achieves better results than BOBYQA when the tolerance is set to $10^{-7}$.
\begin{figure}[h!]
    \centering
    {\includegraphics[width=0.33\textwidth]{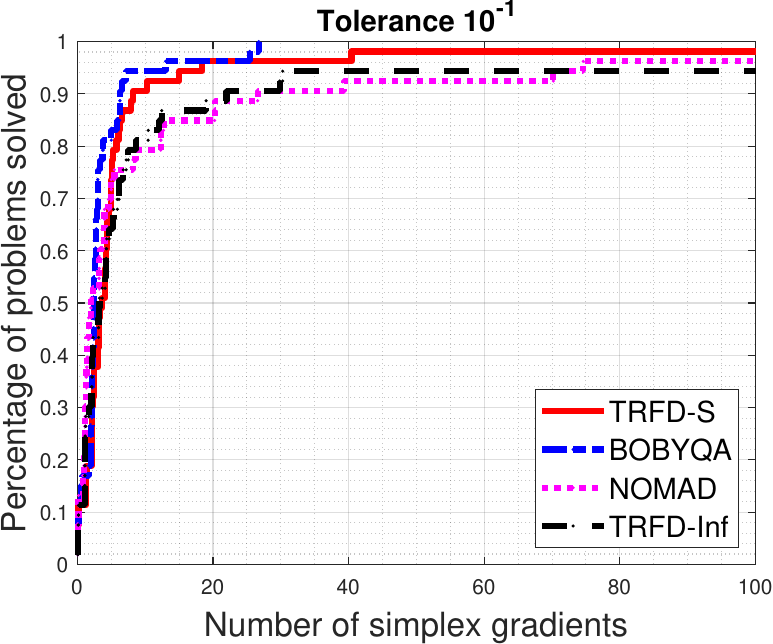}}
    \hspace{5mm}
    {\includegraphics[width=0.33\textwidth]{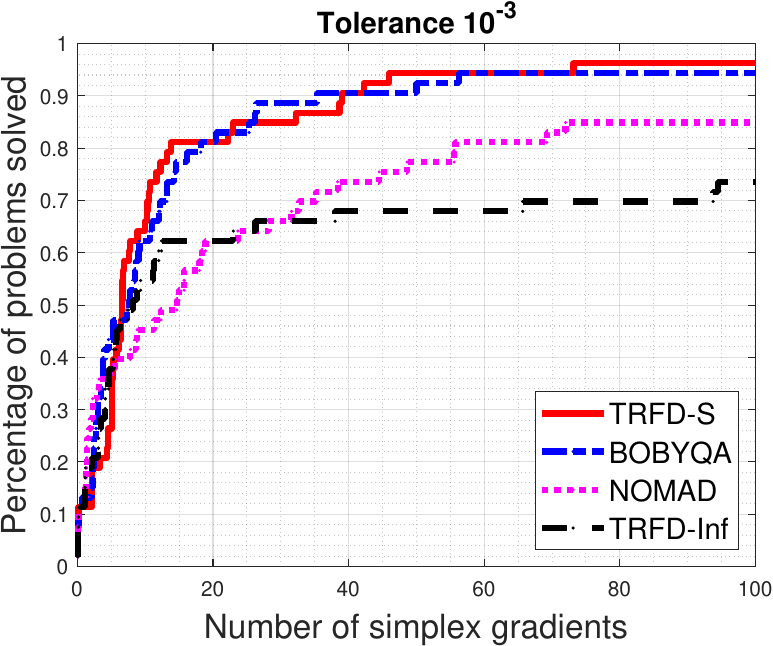}} 
    {\includegraphics[width=0.33\textwidth]{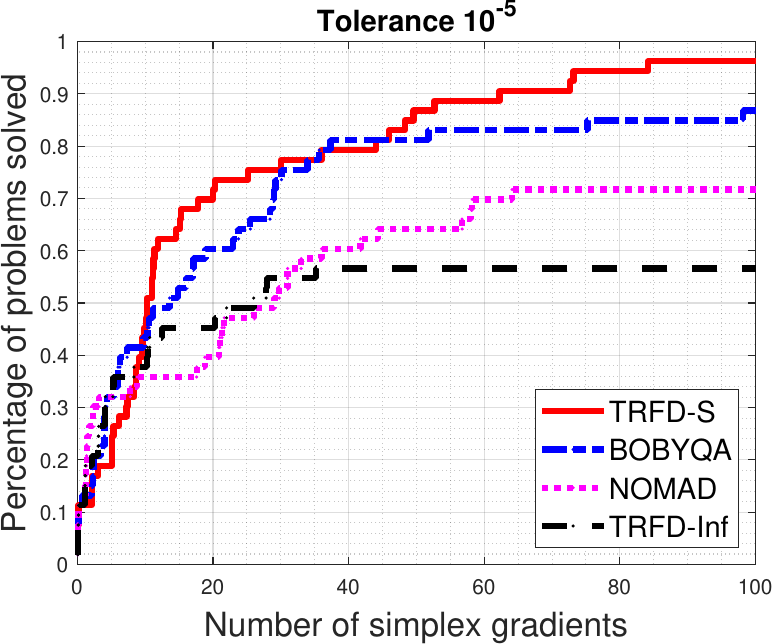}}
    \hspace{5mm}
    {\includegraphics[width=0.33\textwidth]{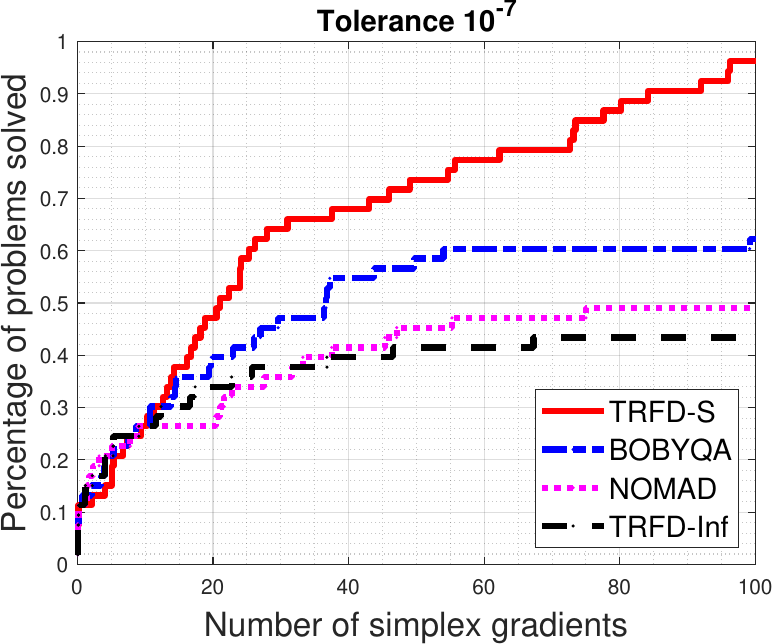}}
    \caption{Data profiles for smooth unrelaxable bound constraints problems.}
    \label{fig:4}
\end{figure}
{\color{black}
\subsubsection{Performance of TRFD-S with other DFO methods with additive noise}

Similarly as in subsection \ref{sec:5:noise}, additional uniform noise was considered with the interval $\left[-\sqrt{3}, \sqrt{3}\right]$ and standard deviation $s.d. \in \{10^{-1}, 10^{-3}, 10^{-5}, 10^{-7}\}$, following the same setup as in \cite[Sections 2 and 4]{shi_nocedal}. Data profiles are shown on Figure \ref{fig:bound_noise}. As we can see, TRFD-S \textcolor{black}{remarkably} outperforms BOBYQA, NOMAD and TRFD-Inf for all levels of noise, \textcolor{black}{with the gap decreasing as we increase the noise level}.
\begin{figure}[h!]
    \centering
    {\includegraphics[width=0.32\textwidth]{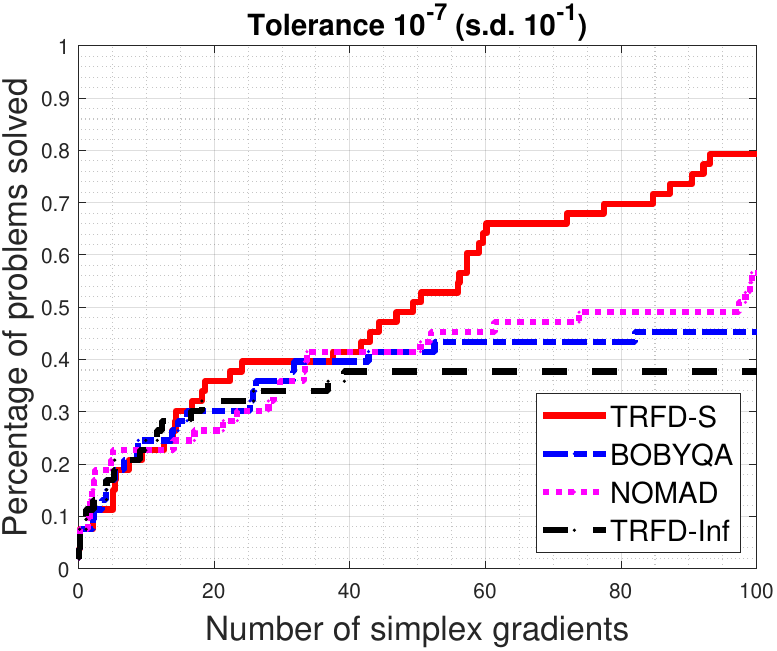}}
    \hspace{5mm}
    {\includegraphics[width=0.32\textwidth]{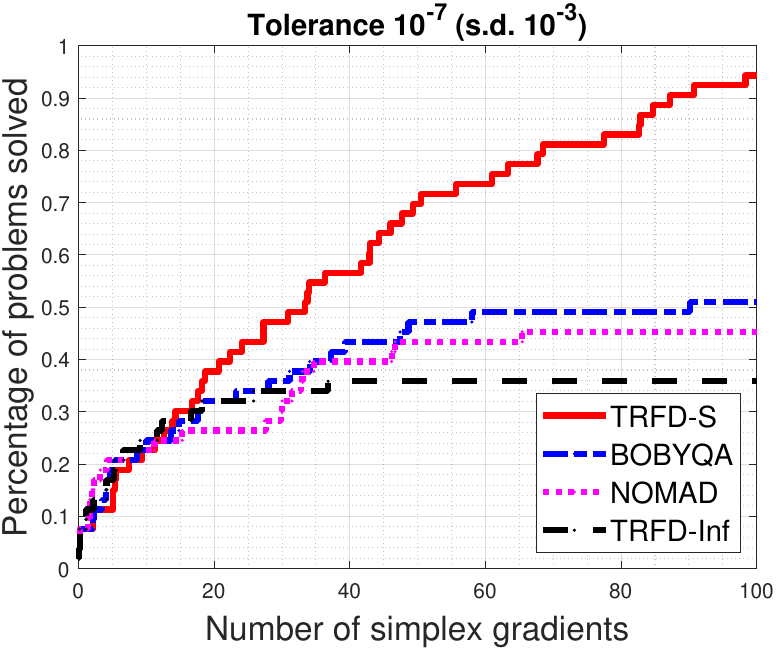}} 
    {\includegraphics[width=0.32\textwidth]{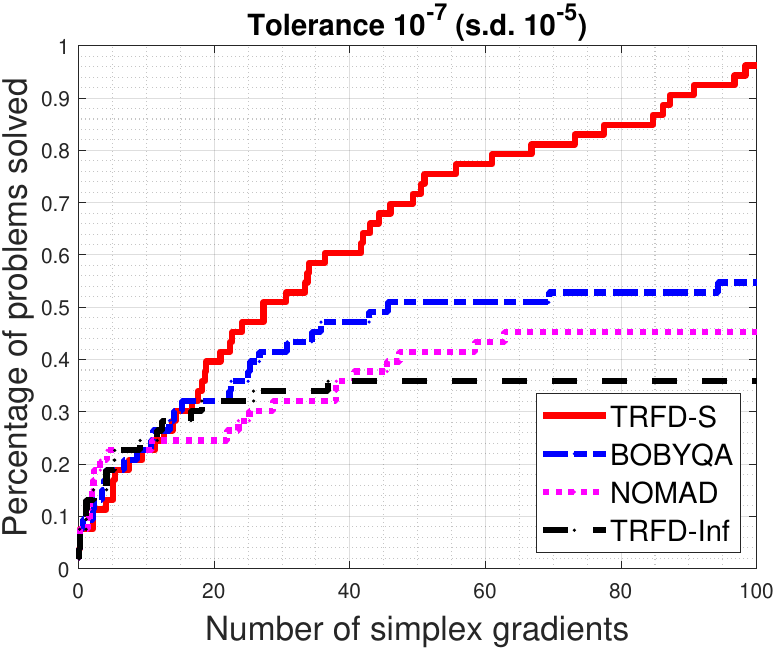}}
    \hspace{5mm}
    {\includegraphics[width=0.32\textwidth]{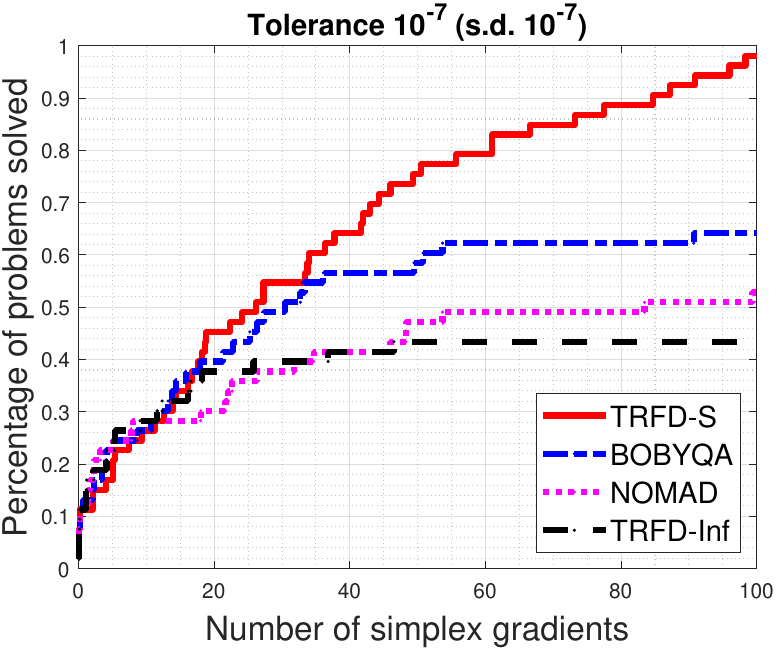}}
    \caption{Data profiles for smooth unrelaxable bound constraints problems with noise.}
    \label{fig:bound_noise}
\end{figure}
}
\subsection{Calibration of an ODE Model}\label{sec:app}

Parameter calibration in differential equation models plays a central role in describing physical, biological, and engineering processes. Although gradient-based optimization methods provide an efficient framework for this task, their practical use is often limited by the complexity of implementing adjoint equations or automatic differentiation techniques. Consequently, many practitioners still favor derivative-free methods for their ease of application (see, e.g., \cite{pumpkin,dengue,covid1,covid2,zika,china}). In this section, we compare the performance of TRFD-S and BOBYQA in calibrating the parameters of the Rosenzweig--MacArthur extension of the Lotka--Volterra Predator-Prey model \cite{rosenzweig,turchin}:
\begin{equation}\label{lotka}
\begin{aligned}
\frac{d\,Y(t)}{dt} &= \zeta\,Y(t)\left(1-\frac{Y(t)}{\theta}\right) - \lambda\,\frac{Y(t)\,Z(t)}{\mu+Y(t)},\\
\frac{d\,Z(t)}{dt} &= \nu\,\frac{Y(t)\,Z(t)}{\mu+Y(t)} - \xi\,Z(t),
\end{aligned}
\end{equation}
where $Y(t)$ and $Z(t)$ are the Preys and Predators densities, respectively. Using the initial conditions $Y(0)=400$ and $Z(0)=20$, we generated a synthetic dataset based on the solutions of the system \eqref{lotka}, which depend on the vector of parameters $x = [\zeta, \theta, \lambda, \mu, \nu, \xi]^T$. The solutions of the system \eqref{lotka} are denoted by $Y(t; x)$ and $Z(t;x)$.

We begin by selecting a ground-truth parameter vector $$x^* = [0.723, 447, 2.88, 21.9, 5.54, 4.99]^T.$$ Using this reference set of parameters, we numerically solved the system with the MATLAB function \texttt{ode45} \textcolor{black}{with default options} to obtain the solutions at a discrete set of time points $\{t_i=0.5i\}_{i=0}^{70}$, yielding the trajectories $\{Y(t_i; x^*)\}_{i=0}^{70}$ and $\{Z(t_i; x^*)\}_{i=0}^{70}$. To simulate observational noise, we perturbed the exact solutions with additive Gaussian noise. Specifically, we defined the synthetic observations as:
\[
\tilde{Y}_i = Y(t_i; x^*) + 10\varepsilon_i, \quad \tilde{Z}_i = Z(t_i; x^*) + 10\varepsilon_i, \quad \varepsilon_i \sim \mathcal{N}(0, 1),
\]
for $i = 0, \ldots, 70$. The resulting datasets $\{(t_i, \tilde{Y}_i)\}_{i=0}^{70}$ and $\{(t_i, \tilde{Z}_i)\}_{i=0}^{70}$ served as the testbed for calibrating $x$. Thus, given both datasets, by denoting $\bar{Y}$ and $\bar{Z}$ as the mean values of the Preys and Predators populations, respectively, we defined the least-square error as:
\begin{equation}\label{eq:obj_ls}
    f(x) \equiv \frac{1}{\bar{Y}^2}\sum_{i=0}^{70} \left(Y(t_i;x)-\tilde{Y}_i\right)^2 + \frac{1}{\bar{Z}^2}\sum_{i=0}^{70} \left(Z(t_i;x)-\tilde{Z}_i\right)^2.
\end{equation}
Given the objective function \eqref{eq:obj_ls}, we created a set of 171 problems defined by the initial point $$x_0 = i\,[0.1, 100, 1, 10, 1, 1]^T, \quad i=1,1.5,...,9,9.5,10$$ and the unrelaxable bounds $$\ell = [0.001, 0.001, \\ 0.001, 0.001, 0.001, 0.001]^T\quad\text{and}\quad u = x_0+j\,(x_0-\ell),$$ 
for $j=1,1.5,...,4,4.5,5$. We justify the choice of the vector $[0.1, 100, 1, 10, 1, 1]^T$ in $x_0$ by the order of magnitude that the parameters should have, for intrinsic biological reasons. Indeed, setting $x_0$ with components that are out of the right order of magnitude could lead to numerical issues with the ODE solver, independently of the method that is used.

For both methods, we used the same setup as in subsection \ref{perf_dfo_bc} and gave a budget of 100 simplex gradients. As we can see in Figure \ref{fig:ode}, \textcolor{black}{for this application} TRFD-S outperforms BOBYQA for all tolerances presented.
\vspace{-0.3cm}
\begin{figure}[h!]
    \centering
    {\includegraphics[width=0.32\textwidth]{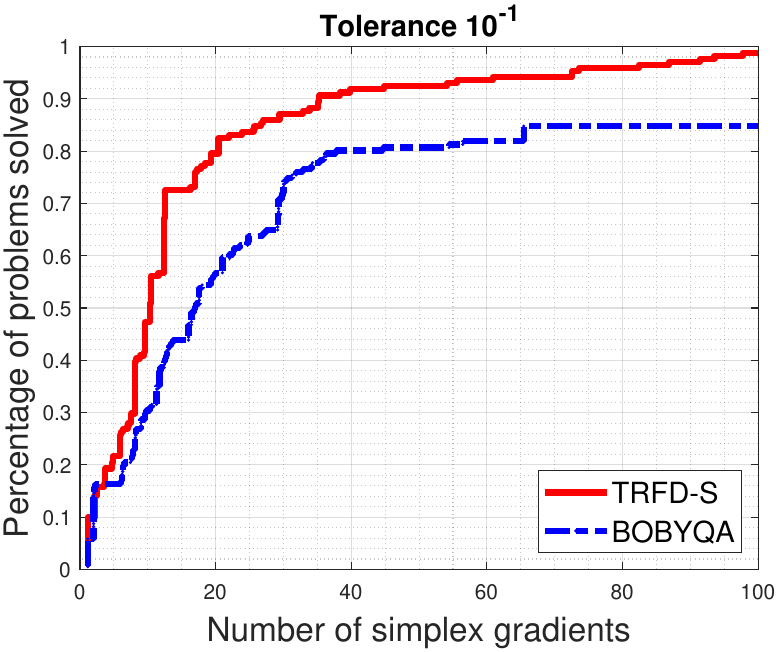}}
    \hspace{5mm}
    {\includegraphics[width=0.32\textwidth]{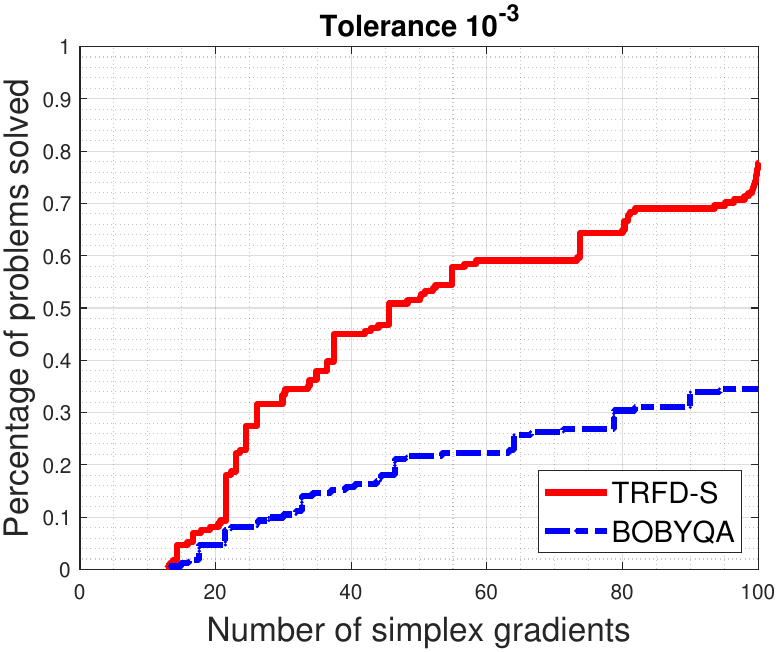}} 
    {\includegraphics[width=0.32\textwidth]{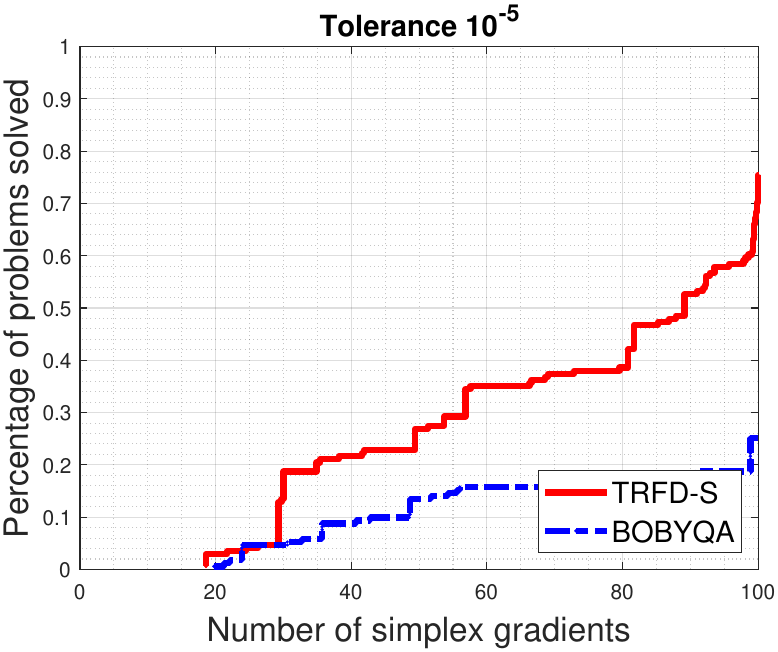}}
    \hspace{5mm}
    {\includegraphics[width=0.32\textwidth]{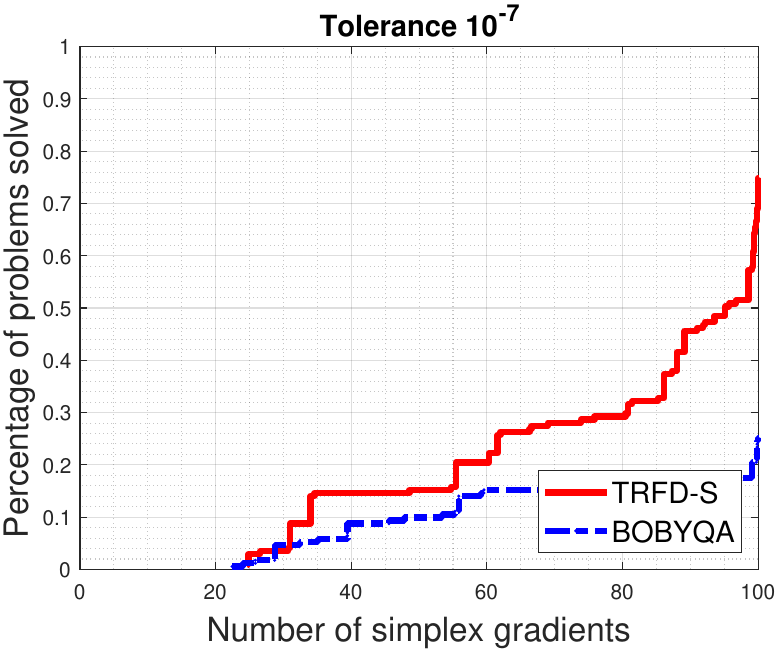}}
    \caption{Data profiles for the calibration of an ODE model.}
    \label{fig:ode}
\end{figure}

\section{Conclusion}\label{sec:5}
This work presented TRFD-S, a derivative-free trust-region method based on finite-difference gradient approximations designed for smooth convexly constrained optimization problems. TRFD-S relies on second-order models, which are assumed to produce at least a Cauchy decrease. Moreover, the method does not require the computation of the approximate stationarity measure $\eta_{\Delta_{\mathrm{max}}}(x_k)$. In this setting, worst-case evaluation complexity bounds were established. Specifically, for nonconvex problems, it was shown that TRFD-S requires at most $\mathcal{O}\left(n\left(\frac{L}{\sigma}\epsilon\right)^{-2}\right)$ function evaluations to reach an $\left(\frac{L}{\sigma}\epsilon\right)$-approximate stationary point of $f\left(\,\cdot\,\right)$ in $\Omega$, while a bound of $\mathcal{O}\left(n\left(\frac{L}{\sigma}\epsilon\right)^{-1}\right)$ was proved for convex problems to find an $\left(\frac{L}{\sigma}\epsilon\right)$-approximate minimizer of $f\left(\,\cdot\,\right)$ in $\Omega$. Also, a bound of $\mathcal{O}\left(n\log\left(\left(\frac{L}{\sigma}\epsilon\right)^{-1}\right)\right)$ was obtained when $\Omega=\mathbb{R}^n$ for Polyak-Lojasiewicz functions. In addition, a simple adaptation of the method was proposed for problems with unrelaxable bound constraints. Numerical results were also presented \textcolor{black}{without and with additive uniform noise}, showing the superior performance of TRFD-S on smooth problems compared with TRFD \cite{grapiglia}, DFQRM \cite{grapiglia2}, and NOMAD \cite{audet}, as well as its competitive behavior relative to NEWUOA \cite{powell,zhang_2023} and BOBYQA \cite{bobyqa,zhang_2023}.

\section*{Acknowledgments}

We are extremely grateful to Zaikun Zhang for his assistance with the Prima package.

\appendix

\section{Sufficient Decrease via the Generalized Cauchy Step}
\label{sec:appendix}

For clarity, we omit the iteration index $k$ in this section.
\\
\\
Let us consider the following definition:
\begin{definition}
    Given $\Omega$, $x\in \Omega$, $0<\Delta \leq \Delta_{\mathrm{max}}$, $g\in \mathbb{R}^n$, $H \in \mathbb{R}^{n\times n}$ and $t>0$, let us denote $d = P_{\Omega}(x-tg)-x$. If the step $d$ satisfies both
    \begin{equation}
        \|d\| \leq \Delta \quad \text{and} \quad m(d) \leq  m(0) + \kappa_u \langle g, d \rangle \label{both}
    \end{equation}
    and at least
        \begin{align}
                \|d\| &\geq \kappa_f\Delta \label{one_1}\\
                \hspace{1mm} &\text{or} \notag\\
                m(d) &\geq m(0) + \kappa_\ell \langle g, d \rangle \label{one_2}\\
                &\text{or} \notag\\
                \|P_{\mathcal{T}(x+d)}(-g)\| &\leq \kappa_e \frac{|\langle g, d \rangle|}{\Delta}, \label{one_3}
        \end{align}
    where
    \begin{equation}\label{kappas}
        0 < \kappa_u < \kappa_\ell < 1, \quad \kappa_f \in (0,1), \quad \kappa_e \in (0, \frac{1}{2}),
    \end{equation}
    then we denote $d^{GC} := d$, and $d^{GC}$ is called the Generalized Cauchy step.
\end{definition}
Moreover, we define
\begin{itemize}
    \item The normal cone of $\Omega$ at $x$: $$\mathcal{N}(x):=\{y \in \mathbb{R}^n: \langle y, z-x \rangle \leq 0, \; \forall z \in \Omega\}.$$
    \item The polar of $\mathcal{N}(x)$: $$\mathcal{N}^0(x):=\{y \in \mathbb{R}^n: \langle y, z \rangle \leq 0, \; \forall z \in \mathcal{N}(x)\}.$$
    \item The tangent cone of $\Omega$ at $x$: $\mathcal{T}(x)=\mathcal{N}^0(x)$.
    \item The Moreau decomposition of $x$ relative to $y$: $$x=P_{\mathcal{T}(y)}(x) + P_{\mathcal{N}(y)}(x).$$
\end{itemize}
\noindent The following lemma provides two properties on the approximate stationarity measure.
\begin{lemma}\label{lem:ab}
    Suppose that A1 holds. Given $r>0$, let $\eta_r$ be defined by \eqref{eta}. Then, \\ \\ (a) $\eta_r(x) \geq 0, \quad \forall x \in \Omega$; \\ (b) $\eta_r(x) \leq \frac{|\langle g,s \rangle|}{r} + 2\|P_{\mathcal{T}(x+s)}(-g)\|, \quad \forall s \in (\Omega-\{x\})\cap B(0,r).$
\end{lemma}
\begin{proof}
    The proof of (b) is an adaptation of Theorem 12.1.5 (iii) in \cite{toint}.
    \vspace{2mm}
    \\
    Clearly, $$\min_{s \in \Omega-\{x\} \atop \|s\|\leq r} \langle g,s \rangle \leq \langle g, s \rangle, \quad \forall s \in (\Omega-\{x\})\cap B[0,r].$$
    In particular,
    \begin{equation}\label{gs}
        \min_{s \in \Omega-\{x\} \atop \|s\|\leq r} \langle g,s \rangle \leq 0.
    \end{equation} By \eqref{gs} and the definition of $\eta_r$ in \eqref{eta}, we conclude that (a) holds. To prove (b), let $s^*\in (\Omega-\{x\})\cap B[0,r]$ be a solution of $$\min_{s \in \Omega-\{x\} \atop \|s\|\leq r} \langle g,s \rangle,$$
    and consider some $s \in (\Omega-\{x\})\cap B[0,r]$.
\noindent On one hand, by applying the Moreau decomposition to $-g$ at $x+s$, we have
    $$-g = P_{\mathcal{T}(x+s)}(-g)+P_{\mathcal{N}(x+s)}(-g).$$
    So, we get
    \begin{align}
        \langle g, s^*-s \rangle &= \langle -P_{\mathcal{T}(x+s)}(-g) - P_{\mathcal{N}(x+s)}(-g), s^* - s \rangle \nonumber\\ &= - \langle P_{\mathcal{T}(x+s)}(-g), s^* - s \rangle - \langle P_{\mathcal{N}(x+s)}(-g), s^* - s \rangle. \label{g_s_star}
    \end{align}
    Moreover, $$s^*-s = (x+s^*) - (x+s) \in \mathcal{N}^0(x+s),$$ since $x+s^* \in \Omega$, by definition of $s^*$. Therefore,
    \begin{equation}\label{s_star}
    \langle s^*-s, z \rangle \leq 0, \quad \forall z \in \mathcal{N}(x+s).    
    \end{equation}
    Since $P_{\mathcal{N}(x+s)}(-g) \in \mathcal{N}(x+s)$, by \eqref{s_star} we have that \eqref{g_s_star} reduces to
    \begin{equation}\label{PT}
        \langle g, s^*-s \rangle \geq - \langle P_{\mathcal{T}(x+s)}(-g), s^* - s \rangle.
    \end{equation}
    On the other hand, by definition of $s^*$, $$\langle g,s^* \rangle \leq \langle g, s \rangle.$$ Thus, 
    \begin{equation}\label{s_hat}
        \langle g,s^*-s \rangle \leq 0, \quad \forall s \in (\Omega-\{x\})\cap B[0,r].
    \end{equation} As a result, by combining \eqref{s_hat} with \eqref{PT}, we have
    \begin{align}
        | \langle g, s^*-s \rangle | &= -  \langle g, s^*-s \rangle \leq \langle P_{\mathcal{T}(x+s)}(-g), s^* - s \rangle \leq |\langle P_{\mathcal{T}(x+s)}(-g), s^* - s \rangle| \\
        &\leq \|P_{\mathcal{T}(x+s)}(-g)\| \|s^*-s\| \notag \leq \|P_{\mathcal{T}(x+s)}(-g)\| \left(\|s^*\|+\|s\|\right) \notag\\
        &\leq \|P_{\mathcal{T}(x+s)}(-g)\| 2r. \label{p2r}
    \end{align}
    Then, by the definition of $\eta_r(x)$ in \eqref{eta}, (a) and \eqref{p2r}, we have
    \begin{align*}
    \eta_r(x) &= \frac{1}{r}(- \langle g,s^*\rangle) = \left|\frac{1}{r}(- \langle g,s^*\rangle)\right| = \frac{1}{r} |\langle g,s^*\rangle| = \frac{1}{r} \left|\langle g, s^* -s+s \rangle\right| \\
    &= \frac{1}{r}|\langle g, s^* -s\rangle + \langle g, s \rangle| \leq \frac{1}{r}\big(|\langle g, s^* -s\rangle| + |\langle g, s \rangle|\big) \\
    &\leq 2\|P_{\mathcal{T}(x+s)}(-g)\| + \frac{|\langle g,s\rangle|}{r},
    \end{align*}
    which shows that (b) is true.
\end{proof}
\noindent Additionally, let us consider the following definition:
\begin{definition}
     Given $x \in \Omega$ and $g\in \mathbb{R}^n$, the projected approximated gradient path is defined as the set $\{p(t,x) := P_{\Omega}(x-tg), \, \forall t \geq 0\}$.
\end{definition}
\noindent The following lemma states that each element of the projected approximated gradient path is the solution of a particular optimization problem.
\begin{lemma}\label{lem:gs}
    Suppose that A1 holds. Given $x\in \Omega$, $g \in \mathbb{R}^n$ and $t\geq 0$, the step $p(t,x)-x$ is a solution of the problem
    \begin{equation}\label{g,s}
        \min_{s \in \Omega-\{x\} \atop \|s\| \leq \|p(t,x)-x\|} \langle g, s \rangle.
    \end{equation}
\end{lemma}
\begin{proof}
    This proof is an adaptation of Theorem 12.1.4 in \cite{toint}.
    \\
    \\
    Let us denote
    \begin{equation}\label{u^*}
        u^*=p(t,x)-x
    \end{equation} as the solution of
    \begin{equation}\label{tg_u}
        \min_{u \in \Omega-\{x\}}\|(x-tg)-(x+u)\|^2 = \min_{u\in\Omega-\{x\}}\|tg+u\|^2.
    \end{equation}
    \textbf{Case I}: $t=0$.
    \vspace{2mm}
    \\
    In this case, we have $p(t,x)=x$. Moreover, $t=0$ implies that the feasible set of \eqref{g,s} is the singleton $\{0\}$. Thus, the step $p(t,x)-x$ is the solution of \eqref{g,s}.
    \\
    \\
    \textbf{Case II}: $t>0$.
    \vspace{2mm}
    \\
    On one hand, let us rewrite \eqref{g,s} as
    \begin{equation}\label{g,s_ref}
        \min_{s \in \Omega-\{x\} \atop \|s\|^2 \leq \|p(t,x)-x\|^2} \langle g, s \rangle.
    \end{equation}
    Then, the KKT conditions of \eqref{g,s_ref} (see Theorem 3.2.11 in \cite{toint}) are given by
    \begin{align}
        \langle g + 2{\lambda} s^*, z-(x+s^*) \rangle &\geq 0, \quad \forall z \in \Omega, \label{kkt_1}\\
        {\lambda} &\geq 0, \label{kkt_2}\\ {\lambda}\left(\|p(t,x)-x\|^2-\|s^*\|^2\right)&=0, \label{kkt_3}
    \end{align}
    where $s^*$ is a solution of \eqref{g,s}. On the other hand, consider the necessary condition for $u^*$ to be a solution of \eqref{tg_u}, i.e.,
    \begin{equation*}
        \langle 2(tg+u^*), z-(x+u^*)\rangle \geq 0, \quad \forall z \in \Omega.
    \end{equation*}
    Dividing by $2t>0$ on both sides, we get 
    \begin{equation*}
       \left\langle g + \frac{u^*}{t}, z-(x+u^*)\right\rangle \geq 0, \quad \forall z \in \Omega.
    \end{equation*}
    Notice that if
    \begin{equation*}
        s^*=u^* \quad \text{and} \quad \lambda = \frac{1}{2t} > 0,
    \end{equation*}
    then \eqref{kkt_1} and \eqref{kkt_2} are satisfied. Moreover, \eqref{kkt_3} is also satisfied, by the definitions of $u^*$ in \eqref{u^*} and of $s^*$. Since \eqref{g,s_ref} is a convex problem, the KKT conditions are sufficient. Thus, we conclude that the step $s^*=p(t,x)-x$ is a solution of \eqref{g,s}.
\end{proof}
\noindent The next lemma gives the relation between two approximate stationarity measures that have a different reference radius.
\begin{lemma}\label{lem:last}
    Suppose that A1 holds. Given $x\in \Omega$, $g\in \mathbb{R}^{n}$ and $0<r_1\leq r_2$, we have $$\eta_{r_1}(x)\geq \eta_{r_2}(x).$$
\end{lemma}
\begin{proof}
    See Lemma 2.12 in \cite{grapiglia} with $p=2$, $F: \mathbb{R}^n \to \mathbb{R}$ and $h(z)=z, \; \forall z \in \mathbb{R}$.
\end{proof}
\noindent Given the quadratic model $m(d)=f(x)+\langle g, d\rangle + \frac{1}{2} \langle Hd, d \rangle$, the following theorem shows that the Generalized Cauchy step $d^{GC}$ gives a sufficient decrease inside this approximation.
\begin{theorem}\label{gen_cauchy}
    Suppose that A1 holds and let $d^{GC}$ be the Generalized Cauchy step. Then,
    \begin{equation}\label{sgc}
        m(0) - m(d^{GC}) \geq \kappa \eta_{\Delta_{\mathrm{max}}}(x) \min\left\{\Delta, \frac{\eta_{\Delta_{\mathrm{max}}}(x)}{\|H\|}\right\},
    \end{equation}
    where $\kappa := \min\left\{2\kappa_u(1-\kappa_\ell), \frac{\kappa_u \kappa_f}{2}\right\}$, and $\eta_{\Delta_{\mathrm{max}}}(x)$ is defined by \eqref{eta} with $r=\Delta_{\mathrm{max}}$.
\end{theorem}
\begin{proof}
    This proof is inspired by Theorem 12.2.2 in \cite{toint}.
    \vspace{0.2cm}
    \\
    \noindent By the second inequality in \eqref{both}, Lemma \ref{lem:gs} and \eqref{eta}, we have
    \begin{align}
        m(0)-m(d^{GC}) &\geq - \kappa_u \langle g, d^{GC} \rangle = \kappa_u \|d^{GC}\|\left(-\frac{1}{\|d^{GC}\|} \langle g, d^{GC} \rangle \right) = \kappa_u \|d^{GC}\|\eta_{\|d^{GC}\|}(x).\label{kappa_eta}
    \end{align}
    Now, let us consider the three following cases.
    \vspace{2mm}
    \\
\textbf{Case I}: $\|d^{GC}\| \geq \kappa_f\Delta$.
\vspace{2mm}
\\
By \eqref{kappa_eta}, \eqref{one_1} and Lemma \ref{lem:last}, it follows that
\begin{align}
    m(0) - m(d^{GC}) &\geq \kappa_u \kappa_f\Delta\eta_{\|d^{GC}\|}(x) \geq \kappa_u \kappa_f\Delta\eta_{\Delta}(x) \geq \kappa_u \kappa_f\Delta\eta_{\Delta_{\mathrm{max}}}(x).\label{caseI}
\end{align}
\textbf{Case II}: $m(d^{GC}) \geq m(0) + \kappa_\ell \langle g, d^{GC} \rangle$.
\vspace{2mm}
\\
From the model definition, we have
\begin{equation}\label{sHs}
    \langle Hd, d^{GC} \rangle = 2 \left(m(d^{GC})-m(0)-\langle g, d^{GC} \rangle\right).
\end{equation}
Then, by \eqref{sHs}, \eqref{one_2}, Lemma \ref{lem:gs} and \eqref{eta}, it follows that
\begin{align*}
    \|H\| &= \frac{\|H\|\|d^{GC}\|^2}{\|d^{GC}\|^2} \geq \frac{\langle Hd, d^{GC} \rangle}{\|d^{GC}\|^2} \geq \frac{2 (m(d^{GC})-m(0)-\langle g, d^{GC} \rangle)}{\|d^{GC}\|^2} \geq \frac{2 (\kappa_\ell \langle g,d^{GC} \rangle - \langle g,d^{GC} \rangle)}{\|d^{GC}\|^2} \\
    &= \frac{2(\kappa_\ell-1)\langle g,d^{GC} \rangle}{\|d^{GC}\|^2} = \frac{2(1-\kappa_\ell)}{\|d^{GC}\|}\left(-\frac{1}{\|d^{GC}\|}\langle g,d^{GC} \rangle\right) = \frac{2(1-\kappa_\ell)}{\|d^{GC}\|}\eta_{\|d^{GC}\|}(x),
\end{align*}
which is equivalent to
\begin{equation*}
    \|d^{GC}\| \geq 2(1-\kappa_{\ell})\frac{\eta_{\|d^{GC}\|}(x)}{\|H\|}.
\end{equation*}
Multiplying by $\eta_{\|d^{GC}\|}(x)$ on both sides, we obtain
\begin{equation}\label{equiv_kappa_s}
    \|d^{GC}\|\eta_{\|d^{GC}\|}(x) \geq 2(1-\kappa_\ell)\frac{\eta_{\|d^{GC}\|}(x)^2}{\|H\|}.
\end{equation}
Using \eqref{kappa_eta}, \eqref{equiv_kappa_s} and Lemma \ref{lem:last}, we get
\begin{equation}
    m(0) - m(d^{GC}) \geq \kappa_u2(1-\kappa_\ell)\frac{\eta_{\|d^{GC}\|}(x)^2}{\|H\|} \geq 2\kappa_u(1-\kappa_\ell)\frac{\eta_{\Delta_{\mathrm{max}}}(x)^2}{\|H\|}.\label{caseII}
\end{equation}
\textbf{Case III}: $\|P_{\mathcal{T}(x+d^{GC})}(-g)\| \leq \kappa_e \frac{|\langle g, d^{GC} \rangle|}{\Delta}$ and $\|d^{GC}\| < \kappa_f \Delta$.
\vspace{2mm}
\\
From (b) in Lemma \ref{lem:ab}, we have
\begin{equation*}
    \eta_{\kappa_f\Delta}(x) \leq \frac{|\langle g,d^{GC} \rangle|}{\kappa_f\Delta} + 2\|P_{\mathcal{T}(x+d^{GC})}(-g)\|, \quad \forall d^{GC} \in (\Omega-\{x\})\cap B(0,\kappa_f\Delta).
\end{equation*}
Then, by \eqref{one_3}, the definitions of $\kappa_e,\kappa_f$ in \eqref{kappas}, Lemma \ref{lem:gs}, \eqref{eta} and (a) in Lemma \ref{lem:ab}, it follows that
\begin{align*}
    \eta_{\kappa_f\Delta}(x) &\leq \frac{|\langle g,d^{GC} \rangle|}{\kappa_f\Delta} + \frac{2\kappa_e|\langle g,d^{GC} \rangle|}{\Delta} \leq \frac{2|\langle g,d^{GC} \rangle|}{\kappa_f\Delta} = \frac{2\|d^{GC}\|}{\kappa_f\Delta}\left|\frac{1}{\|d^{GC}\|}\langle g,d^{GC} \rangle\right| \\
    &= \frac{2\|d^{GC}\|}{\kappa_f\Delta}\left|-\eta_{\|d^{GC}\|}(x)\right| = \frac{2\|d^{GC}\|}{\kappa_f\Delta}\eta_{\|d^{GC}\|}(x),
\end{align*}
which is equivalent to
\begin{equation}\label{equiv_new}
    \|d^{GC}\|\eta_{\|d^{GC}\|}(x) \geq \frac{\kappa_f}{2}\Delta\eta_{\kappa_f\Delta}(x).
\end{equation}
By \eqref{kappa_eta}, \eqref{equiv_new} and Lemma \ref{lem:last}, it follows that
\begin{equation}
    m(0) - m(d^{GC}) \geq \frac{\kappa_u \kappa_f}{2} \Delta \eta_{\kappa_f\Delta}(x) \geq \frac{\kappa_u \kappa_f}{2} \Delta \eta_{\Delta_{\mathrm{max}}}(x).\label{caseIII}
\end{equation}
Finally, by combining \eqref{caseI}, \eqref{caseII}, \eqref{caseIII} and the definition of $\kappa$, we obtain
\begin{align*}
    m(0) - m(d^{GC}) &\geq \min\left\{\kappa_u \kappa_f\Delta\eta_{\Delta_{\mathrm{max}}}(x), 2\kappa_u(1-\kappa_\ell)\frac{\eta_{\Delta_{\mathrm{max}}}(x)^2}{\|H\|}, \frac{\kappa_u \kappa_f}{2} \Delta \eta_{\Delta_{\mathrm{max}}}(x)\right\} \\
    &= \min\left\{2\kappa_u(1-\kappa_\ell)\frac{\eta_{\Delta_{\mathrm{max}}}(x)^2}{\|H\|}, \frac{\kappa_u \kappa_f}{2} \Delta \eta_{\Delta_{\mathrm{max}}}(x)\right\} \\
    &\geq \kappa \min\left\{\frac{\eta_{\Delta_{\mathrm{max}}}(x)^2}{\|H\|}, \Delta \eta_{\Delta_{\mathrm{max}}}(x)\right\} \\
    &= \kappa \eta_{\Delta_{\mathrm{max}}}(x) \min\left\{\frac{\eta_{\Delta_{\mathrm{max}}}(x)}{\|H\|}, \Delta\right\},
\end{align*}
which implies that \eqref{sgc} is true.
\end{proof}

\bibliographystyle{siamplain}

\end{document}